\documentclass[a4paper,12pt]{article}
\usepackage{latexsym,amsmath,amsthm,amssymb}
\usepackage{a4wide}
\usepackage{hyperref}
\usepackage{marginnote}
\usepackage{color}
\usepackage{cite}
\hypersetup{
pdftitle={Second order CKN, Stability estimates}   
pdfauthor={Van Hoang},
colorlinks = true,
linkcolor = magenta,
citecolor = blue,
}

\theoremstyle{plain}
\newtheorem{theorem}{Theorem}[section]

\newtheorem{proposition}[theorem]{Proposition}
\newtheorem{lemma}[theorem]{Lemma}
\newtheorem{corollary}[theorem]{Corollary}

\theoremstyle{definition}

\theoremstyle{remark}

\allowdisplaybreaks

\renewcommand{\thefootnote}{\arabic{footnote}}
\def\R{\mathbb R}% tap so thuc
\def\N{\mathbb N}% tap so tu nhien
\def\Z{\mathbb Z}% tap so nguyen

\def\al{\alpha}% alpha
\def\be{\beta}% beta
\def\ga{\gamma}% gamma
\def\de{\delta}% delta
\def\De{\Delta} % Delta

\def\lam{\lambda}% lambda
\def\ep{\epsilon}% epsilon
\def\na{\nabla}% nabla
\def\pa{\partial}% dao ham rieng
\def\lt{\left}% trai
\def\rt{\right}% phai

\def\i0i{\int_0^\infty}

\def\irn{\int_{\R^n}}

\numberwithin{equation}{section}
\usepackage{color}

\title{The sharp second order Caffareli-Kohn-Nirenberg inequality and stability estimates for the sharp second order uncertainty principle}
\author{Anh Tuan Duong\footnote{School of Applied Mathematics and Informatics, Hanoi University of Science and Technology, 1 Dai Co Viet, Hai Ba Trung, Ha noi, Viet Nam.}\, and Van Hoang Nguyen\footnote{Department of Mathematics, FPT University, Ha Noi, Viet Nam.} 
}

\begin{document}
\maketitle

%% Classification and key words; note that the 2010 classification is used:

\renewcommand{\thefootnote}{}

\footnote{Email: \href{mailto: Anh Tuan Duong<tuan.duonganh@hust.edu.vn>}{tuan.duonganh@hust.edu.vn}; \href{mailto: Van Hoang Nguyen <vanhoang0610@yahoo.com>}{vanhoang0610@yahoo.com} and \href{mailto:Van Hoang Nguyen <hoangnv47@fe.edu.vn>}{hoangnv47@fe.edu.vn}.}

\footnote{2010 \emph{Mathematics Subject Classification\text}: 26D10, 46E35, 26D15}

\footnote{\emph{Key words and phrases\text}: Caffarelli--Kohn--Nirenberg inequalities, uncertainty principle, sharp constant, extremal functions, stability estimates}

\renewcommand{\thefootnote}{\arabic{footnote}}
\setcounter{footnote}{0}

\begin{abstract}
In this paper we prove a class of second order Caffarelli-Kohn-Nirenberg inequalities which contains the sharp second order uncertainty principle recently established by Cazacu, Flynn and Lam \cite{CFL2020} as a special case. We also  show the sharpness of our inequalities for several classes of parameters. Finally, we prove a stability version of the  sharp second order uncertainty principle of Cazacu, Flynn and Lam by showing that the difference of both sides of the inequality controls the distance to the set of  extremal functions in $L^2$ norm of gradient of functions.

%Cazacu, Flynn and Lam \cite{CFL2020} proved a sharp second order uncertainty principle
%$$\irn |\Delta u|^2 dx \irn |x|^2 |\na u|^2 dx \geq \Big(\frac{n+2}2\Big)^2 \Big(\irn |\na u|^2 dx\Big)^2.$$
%The sharp constant $(n+2)^2/4$ is attained by Gaussian function. The first aim in this paper is to extend the inequality above to a larger class of the second order Caffarelli-Kohn-Nirenberg inequality. More precisely, we shall prove that the inequality 
%$$\irn \frac{|\Delta u|^2}{|x|^{2\alpha}} dx \irn \frac{|\na u|^2}{|x|^\beta} dx \geq \Big(\frac{n + 4\al -2\ga +2}2\Big)^2 \Big(\irn \frac{|\na u|^2}{|x|^{2\gamma}} dx\Big)^2$$
%holds true for $\al, \beta,\gamma$ satisfying some suitable conditions. We also give the conditions on $\al, \beta, \ga$ such that the constant in the above inequality is sharp and classify all extremal functions. Finally, we prove the stability version of the second order uncertainty principle above by showing that the difference of both sides controls the distance to the set of Gaussian functions both in $L^2$ norm of function and in $L^2$ norm of gradient of functions.
\end{abstract}

\section{Introduction}
The Heisenberg uncertainty principle in quantum mechanics states that the position and the momentum of a given particle cannot both be determined exactly at the same time (see \cite{Heisenberg}). The rigorous mathematical formulation of this principle is established by Kennard \cite{Kennard} and Weyl \cite{Weyl} (who attributed it to Pauli) stating that the function itself and its Fourier
transform cannot be sharply localized at the origin simultaneously. Mathematically, the  Heisenberg-Pauli-Weyl uncertainty principle is described by the following inequality
\begin{equation}\label{eq:HPW}
\irn |\na u|^2 dx \irn |x|^2 |u|^2 dx \geq \frac{n^2}4 \Big(\irn |u|^2 dx\Big)^2
\end{equation}
for any function $u\in H^1(\R^n)$ (the first order Sobolev space in $\R^n$) such that $\irn |x|^2 |u|^2 dx < \infty$. It is well known that the constant $n^2/4$ is sharp and is attained only by Gaussian functions (see \cite{Folland}). 

In \cite{Xia07}, Xia extends the inequality \eqref{eq:HPW} and obtain the following inequality
\begin{equation}\label{eq:Xia}
\frac{n-t\ga}{t} \irn \frac{|u|^t}{|x|^{t\ga}} dx \leq \Big(\irn \frac{|\na u|^p}{|x|^{\al p}} dx\Big)^{\frac1p} \Big(\irn \frac{|u|^{\frac{p(t-1)}{p-1}}}{|x|^{\beta}} dx\Big)^{\frac{p-1}p}
\end{equation}
where $n\geq 2$, $1< p< t$, and $\alpha, \beta, \gamma$ satisfying the following conditions
\begin{equation*}
n -\alpha p >0,\quad n -\beta >0,\quad n -t\gamma >0,
\end{equation*}
and the balanced condition
\begin{equation*}
\gamma = \frac{1 + \alpha}{t} + \frac{\beta}{tp}.
\end{equation*}
Moreover, when $1 + \al -\frac \be r >0$ and
$$n-\beta < \Big(1 + \al -\frac \be r\Big) \frac{p(t-1)}{t-p}$$
then the inequality \eqref{eq:Xia} is sharp and the extremal functions are given by
$$u(x) = (\lam + |x|^{1 + \al -\frac \be r}\Big)^{\frac{1-p}{t-p}},\qquad \lam >0.$$
The inequality \eqref{eq:Xia} is extended to the Riemannian manifolds in \cite{NguyenProcA}, the Finsler manifolds in \cite{HKZ} and the stratified Lie groups in \cite{NguyenAASFM}.

Both the Heisenberg-Pauli-Weyl principle \eqref{eq:HPW} and the Xia inequality \eqref{eq:Xia} belong to a larger class of the first order interpolation inequalities which are called Caffarelli-Kohn-Nirenberg (CKN) inequalities established in \cite{CKN84} to study the Navier-Stokes equation and the regularity of particular { solutions }\cite{CKN82}. The class of CKN inequalities contains many well-known inequalities such as the Sobolev inequality, the Hardy inequality, the Hardy-Sobolev inequality, the Gagliardo-Nirenberg inequality, etc. They play an important role in theory of partial differential equations and have been extensively studied in many settings. { Concerning to the sharp version of CKN inequalities, we refer the reader to the papers \cite{Talenti,Aubin,Lieb,DD1,DD2,CNV,CWang,CC,Costa}.
}

The higher order CKN inequalities were established by Lin \cite{Lin}. In contrast to the first order inequalities, much less is known on the sharp version of the higher order CKN inequalities except the Rellich inequality \cite{Rellich} and the sharp higher order Sobolev inequality \cite{Lieb,Cotsiolis}. In recent paper \cite{CFL2020}, Cazacu, Flynn and Lam proved the following sharp second order uncertainty principle which is a special case of the second order CKN inequalities
\begin{equation}\label{eq:CFL}
\irn |\Delta u|^2 dx \irn |x|^2 |\na u|^2 dx \geq \Big(\frac{n+2}2\Big)^2 \Big(\irn |\na u|^2 dx\Big)^2.
\end{equation}
The constant $(n+2)^2/4$ above is sharp and is attained by Gaussian functions. In fact, the inequality \eqref{eq:CFL} without the sharp constant can be derived from \eqref{eq:HPW} and Cauchy-Schwartz inequality (see the comment in the introduction in \cite{CFL2020}). The inequality \eqref{eq:CFL} is motivated by an open question of Maz'ya \cite{Maz'ya} on finding the sharp constant in \eqref{eq:HPW} when we replace $u$ by a divergence-free vector field $U$. In particular, the inequality \eqref{eq:CFL} answers affirmatively the question of Maz'ya in the case $n=2$.

The first aim in this paper is to extend the inequality \eqref{eq:CFL} to a larger class of parameters in spirit of \eqref{eq:Xia}. For $\alpha, \beta$ satisfying the conditions $n -2\alpha >0$ and $n -\beta >0$, we denote by $H^2_{\al,\beta}(\R^n)$ the second order Sobolev space which is completion of $C_0^\infty(\R^n)$ under the norm 
$$\|u\|_{H^2_{\al,\beta}(\R^n)} = \Big(\irn |\Delta u|^2 |x|^{-2\al} dx + \irn |\na u|^2 |x|^{-\beta} dx\Big)^{\frac12},\; u\in C_0^\infty(\R^n).$$
Then the first main result in this paper reads as follows.

\begin{theorem}\label{Main}
Let $n\geq 1$ and $\alpha\in \R$ satisfy $n -2\al >0$, $n + 2\al >0$ and $n+ 2 + 4\al >0$. Then the following inequality
\begin{equation}\label{eq:CKNalpha}
\irn \frac{|\Delta u|^2}{|x|^{2\alpha}} dx \irn |x|^{2\al} |\na u\cdot x|^2 dx \geq \Big(\frac{n+4\al + 2}2\Big)^2 \Big(\irn |\na u|^2 dx \Big)^2
\end{equation}
holds true for any function $u \in H_{\al,-2-2\al}^2(\R^n)$. Furthermore, if $1+ \al >0$ then the inequality \eqref{eq:CKNalpha} is sharp and is attained by function %all extremal functions are determined up to a dilation and a multiplicative constant by function
$$U_0(x) = \exp\left(-\frac{|x|^{2(1+\al)}}{2(1+\al)}\right).$$
\end{theorem}
{ In particular, when $\alpha=0,$ Theorem \ref{Main} implies the following.
\begin{corollary}\label{Maincor}
	Let $n\geq 1$. Then there holds
	\begin{equation}\label{eq:CKNalphacor}
		\irn |\Delta u|^2 dx \irn |\na u\cdot x|^2 dx \geq \Big(\frac{n+ 2}2\Big)^2 \Big(\irn |\na u|^2 dx \Big)^2 
	\end{equation}
 for any $u\in H^2_{0,-2}(\R^n).$ This inequality is sharp and is attained by the Gaussian function $\exp(-|x|^2/2)$.
\end{corollary}
Evidently, we have $|\na u\cdot x|\leq |\na u| |x|$. Hence, our inequality \eqref{eq:CKNalphacor} is still stronger than \eqref{eq:CFL}. 

Recall that the proof of \eqref{eq:CFL} in \cite{CFL2020} is quite long and complicated. The authors have used the decomposition of function $u$ into spherical harmonic and integral estimates for radial functions.

In order to  prove \eqref{eq:CKNalpha}, we develop a new approach which is completely different with the one of Cazacu, Flynn and Lam. Indeed, our approach is based on establishing a new identity}
$$\irn \frac{|\Delta u + \na u\cdot x |x|^{2\alpha}|^2}{|x|^{2\alpha}} dx = \irn \frac{|\Delta u|^2}{|x|^{2\alpha}} dx + \irn |x|^{2\alpha} |\na u \cdot x|^2  dx - (n-2)\irn |\na u|^2 dx.$$
Then,  using a factorization of $u$ as $u = v U_0$, we are able to show that
$$\irn \frac{|\Delta u + \na u\cdot x |x|^{2\alpha}|^2}{|x|^{2\alpha}} dx \geq 2(n+ 2\alpha) \irn |\na u|^2 dx.$$
Combining two estimates above and a simple minimizing argument, we obtain \eqref{eq:CKNalpha}. More details in the proof are given in Section \S2 below.

For more general on the parameters as in \eqref{eq:Xia}, we obtain the following inequality for radial functions.
\begin{theorem}\label{Mainradial}
Let $n\geq 1$, $t \geq 2$ and $\al, \beta, \ga$ be such that
\begin{equation}\label{eq:khatich}
n -2\alpha >0,\quad n -\beta >0,\quad n -t\gamma >0
\end{equation}
and
\begin{equation}\label{eq:balance}
\gamma = \frac{1 + \alpha}{t} + \frac{\beta}{2t}.
\end{equation}
Then the following inequality 
\begin{equation}\label{eq:2CKN}
\irn \frac{|\Delta u|^2}{|x|^{2\alpha}} dx \irn \frac{|\nabla u|^{2(t-1)}}{|x|^\beta} dx \geq\lt(\frac{{ n}+t(1+2 \alpha -\gamma)}t\rt)^2  \lt(\int_{\R^n} \frac{|\nabla u|^t}{|x|^{t\gamma}}dx\rt)^2
\end{equation}
holds true for any radial function $u \in H^2_{\al,\beta}(\R^n)$. Moreover, under the following conditions
\begin{equation}\label{eq:welldefineu}
(1+ 2\al)(t-2) + 1 + \al -\frac\be 2 >0, \, t< 3 + \al -\frac\be 2,
\end{equation}
\begin{equation}\label{eq:Deltau}
n+ 2\al > 0,
\end{equation}
and
\begin{equation}\label{eq:khatichu}
n -\be < \frac{2(t-1)}{t-2}(1 + \al -\frac\beta 2),
\end{equation}
then the constant $(n+ t(1 +2\al -\ga))^2/t^2$ is sharp and is attained only up to a dilation and a multiplicative constant by the function the form
$$U_1(x) = \int_{|x|}^\infty s^{1 + 2\alpha} \exp\Big(-\frac{s^{1+ \alpha -\frac\beta 2}}{1 + \alpha -\frac\beta2}\Big) ds,$$
if $t =2$, and
$$
U_2(x) =\int_{|x|}^\infty r^{1+ 2\al} \Big(1 + (t-2)\frac{r^{(1+2\al)(t-2) + 1 + \al -\frac\be2}}{(1+2\al)(t-2) + 1 + \al -\frac\be2}\Big)^{\frac1{2-t}} dr,
$$
if $t >2$.
\end{theorem}
{ In the special case where }$\alpha = \beta =0$, { $t=2$} and $\gamma =1/2$, we { recover} the second result of Cazacu, Flynn and Lam in \cite{CFL2020} for radial functions in $H_{0,0}^2(\R^n)$ which is the sharp second order Hydrogen uncertainty principle
\begin{equation}\label{eq:CFL2}
\irn |\Delta u|^2 dx \irn |\na u|^2 dx \geq \frac{(n+1)^2}4 \Big(\irn \frac{|\na u|^2}{|x|} dx\Big)^2.
\end{equation}
In fact, in that paper, Cazacu, Flynn and Lam proved the inequality \eqref{eq:CFL2} for any function $u$ (without radiality assumption) for any $n \geq 5$. The condition $n\geq 5$ appears due to the technical restrictions of decomposing a smooth function into spherical harmonics. They conjectured that the inequality \eqref{eq:CFL2} still holds for $n=2,3,4$. The inequality \eqref{eq:CFL2} is sharp and an extremal function is given by $u(x) = c (1 + a |x|)e^{-a |x|}$ with $c \in \R$ and $a >0$.

{ In the general case of parameters, comparing \eqref{eq:2CKN}} with the first order inequality \eqref{eq:Xia}, the sharp constant changes from $((n-\gamma t)/t)^2$ to $((n +t(1 +2\al -\gamma))/t)^2$. Moreover, to obtain the sharpness and the attainability of constant, we need some more conditions on the parameters (see \eqref{eq:welldefineu} and \eqref{eq:Deltau}). Indeed, these conditions ensure that the function $U_1$ (and $U_2$) is well-defined, and $\Delta U_1$ (and $\Delta U_2$) exists in the distributional sense and belongs to $H_{\al,\beta}^2(\R^n)$.

In the non-radial case, following the approach of Cazacu, Flynn, and Lam by using spherical harmonics, we are able to establish the sharp second order CKN inequalities (see Theorem \ref{MainXiatype} below) which extends \eqref{eq:2CKN} to any function in $H_{\al,\beta}^2(\R^n)$ but with some restrictions on the dimension $n$ (as the case of the second order hydrogen uncertainty principle \eqref{eq:CFL2}). However, this approach works only for $t=2$.

 To state our next result, let us define for $n, \al, \beta, \ga$ as in Theorem \ref{Mainradial} with $t =2$ and $k\geq 0$
\begin{equation}\label{eq:Ak}
A_{n,\al,k}(g) =\int_0^\infty (g'')^2 r^{n+2k -2\al -1} dr +(1 + 2\al) (n+2k-1) \int_0^\infty (g')^2 r^{n +2k -2\al -3} dr
\end{equation}
\begin{equation}\label{eq:Bk}
B_{n,\beta,k}(g) = \int_0^\infty (g')^2 r^{n+2k -\beta -1}dr + \beta k \int_0^\infty g^2 r^{n+2k -\beta -3} dr\qquad\qquad\qquad\qquad\quad\;\;
\end{equation}
\begin{equation}\label{eq:Ck}
C_{n,\ga, k}(g) = \int_0^\infty (g')^2 r^{n+2k -2\ga -1}dr + 2\gamma k \int_0^\infty g^2 r^{n+2k -2\ga -3} dr\qquad\qquad\qquad\qquad\;
\end{equation}
and
\begin{equation}\label{eq:sharpAk}
A_k(n,\al,\beta) = \inf_{g\not\equiv 0}\frac{A_{n,\al,k}(g)B_{n,\beta,k}(g)}{C_{n,\ga,k}(g)^2}
\end{equation}
where the infimum is taken on all function $g\in C^2([0,\infty))$ such that $A_{n,\al,\beta}(g)$ and $B_{n,\be,k}(g)$ are finite.  Our next result is given in the following theorem.
\begin{theorem}\label{MainXiatype}
Let $n \geq 1$ and $\al, \beta,\ga$ satisfy the conditions of Theorem \ref{Mainradial} with $t=2$, then we have
\begin{equation}\label{eq:constant}
\irn \frac{|\Delta f|^2}{|x|^{2\al}} dx\irn \frac{|\na f|^2}{|x|^\beta} dx \geq \min_{k\in \N} A_k(n,\al,\beta)\Big(\irn \frac{|\nabla f|^2}{|x|^{2\ga}}dx\Big)^2,
\end{equation}
for any function $u \in H_{\al,\beta}^2(\R^n)$. Moreover, the constant $\min_{k\in \N} A_k(n,\al,\beta)$ in \eqref{eq:constant} is sharp and satisfies the estimate
\begin{equation}\label{eq:estimateforsharp}
\min_{k\in \N} A_k(n,\al,\beta) \geq \min_{k\in \N}\frac{1 + \min\Big\{0, \frac{4\be k}{(n+ 2k -\beta -2)^2}\Big\}}{1 + \max\Big\{0, \frac{8\ga k}{(n+ 2k -2\ga -2)^2}\Big\}} \Big(\frac{n+ 2 k + 4\al -2\ga +2}{2}\Big)^2.
\end{equation} %Moreover, if $n \geq {\color{red} \text{\rm to fill}}$, $1 + \al -\frac\be2 >0$ and $n + 2\al >0$,  then then the sharp constant is given by
%\begin{equation}\label{eq:sharpnonradial}
%\min_{k} A_k(n,\al,\beta) =A_0(n,\al,\beta) =\Big(\frac{n+ 4\al -2\ga -2}2\Big)^2,
%\end{equation}
%and is attained by the function $U_1$ from Theorem \ref{Mainradial} up to a scaling and a multiplicative constant.
\end{theorem}
%The proof of Theorem \ref{MainXiatype} is based on the following estimate (see the p below)
%\begin{equation*}%\label{eq:keyconstant}
%A_k(n,\al,\beta) \geq \frac{1 + \min\Big\{0, \frac{4\be k}{(n+ 2k -\beta -2)^2}\Big\}}{1 + \max\Big\{0, \frac{8\ga k}{(n+ 2k -2\ga -2)^2}\Big\}} \Big(\frac{n+ 2 k + 4\al -2\ga +2}{2}\Big)^2.
%\end{equation*}
The inequality \eqref{eq:estimateforsharp} was proved by Cazacu, Flynn and Lam in \cite{CFL2020} corresponding to the case $\alpha = \gamma =0, \beta =-2$ and $\alpha = \beta =0, \ga =\frac12$. It plays an important role in their proof of the sharp second uncertainty principle \eqref{eq:CFL} and  the second order hydrogen uncertainty principle \eqref{eq:CFL2}. Indeed, by considering the right-hand side as a function of $k$, they show that the infimum of the right-hand side is attained at $k=0$ in the case $\alpha = \gamma =0, \beta =-2$ for any $n\geq 2$ and in the case $\alpha = \beta =0, \ga =\frac12$ for any $n\geq 5$. This implies \eqref{eq:CFL} and \eqref{eq:CFL2}. We believe that this argument together with \eqref{eq:estimateforsharp} provides the same conclusion for general $\al, \beta, \ga$ when $n$ large enough. Nevertheless, we do not pursue this direction in this paper.

The last aim of this paper is to establish the stability estimates for the sharp second order uncertainty principle \eqref{eq:CFL} of Cazacu, Flynn and Lam. Let us define 
$$
\delta(u) =\frac{\Big(\irn |\Delta u|^2 dx\Big)^{\frac12} \Big(\irn |x|^2 |\na u|^2 dx\Big)^{\frac12}}{\frac{n+2}2 \irn |\na u|^2 dx} - 1
$$
for $u\in H_{0,-2}^2(\R^n)\setminus \{0\}$ and $\delta(0) =0$. Evidently, we always have $\delta(u) \geq 0$. Establishing the stability estimates for the sharp second order uncertainty principle \eqref{eq:CFL} means that we use $\delta(u)$ to control the distance from $u$ to the set of { extremal }functions, i.e., the set

$$
E = \Big\{ c\, e^{-a |x|^2}\, :\, c\in \R, a >0\Big\}.
$$
More precisely, we will prove the following result.
%In fact, we shall prove two stability estimates for \eqref{eq:CFL}. The first stability result uses the distance concerning to the $L^2$ norm of gradient and reads as follows.

\begin{theorem}\label{Stability}
Given $n\geq 2$. Then the following inequality
\begin{equation}\label{eq:Stabestimate}
\delta(u) \geq \frac{1}{384(n+2)} \inf\Big\{\frac{\irn |\na u -\na \varphi|^2 dx}{\irn |\na u|^2 dx} \, :\, \varphi \in E, \irn |\na u|^2 dx = \irn |\na \varphi|^2 dx \Big\},
\end{equation}
holds true for any $u\in H_{0,-2}^2(\R^n)\setminus\{ 0\}$.
\end{theorem}

%The second stability result concerns to the $L^2$ norm of function and is stated as follows
%\begin{theorem}\label{Stabilitynew}
%Given $n\geq 2$. Then the following inequality
%\begin{equation}\label{eq:Stabestimatenew}
%\delta(u) \geq \frac{C}{16n(n+2)} \inf\Big\{\frac{\irn |u -v|^2 dx}{\irn | u|^2 dx} \, :\, v \in E, \irn |u|^2 dx = \irn |v|^2 dx \Big\},
%\end{equation}
%holds true for any $u\in H_{0,-2}^2(\R^n)\setminus\{ 0\}$ with $C = \min\{n/4,1\}$.
%\end{theorem}
In recent years, there { has }been an enourmous attention on establishing the stability version of the sharp inequalities in analysis and geometry, especially the Sobolev type inequality. The question on the stability estimate for the sharp Sobolev inequality was posed by Br\'ezis and Lieb \cite{BL}. This question was affirmatively answered by Bianchi and Egnell \cite{BE} for functions in $H^1(\R^n)$ by exploiting the Hilbert structure of this space. For the case of $W^{1,p}(\R^n)$ with $p\not=2$, the stability estimates for the Sobolev inequality were established in \cite{CianchiBV,CFMPLp,FiNeu,Neu,FMP,Seuffert,FMPadv,CFW}. We also refer the readers to the papers \cite{NguyenGN,CFL14,DT,DT1,FIL,BDNS} for the stability version of the Gagliardo-Nirenberg inequality and logarithmic Sobolev inequality. { In recent paper} \cite{MV}, McCurdy and Venkatraman exploit the approach of Bianchi and Egnell \cite{BE} to prove a stability version of the classical Heisenberg-Pauli-Weyl inequality \eqref{eq:HPW}
$$\irn |\na u|^2 dx \irn |x|^2 |u|^2 dx -\frac{n^2}4 \Big(\irn |u|^2 dx\Big)^2 \geq { \tilde{C}} \inf_{v\in E} \irn |u -v|^2 dx$$
for any $u \in H^1(\R^n)$ with $\irn |u|^2 dx =1$ where $\tilde{C} >0$ in an implicit constant depending on $n$. In \cite{Fathi}, Fathi gave a new and simple proof of the above inequality by using Poincar\'e inequality for Gaussian measure. In \cite{NguyenUP}, the first author establishes the stability version of the inequality \eqref{eq:Xia} generalized the result of McCurdy and Venkatraman to a larger family of parameters.

Let us finish this introduction by some comment on the proofs of Theorem \ref{Stability}. First, we prove an improvement of \eqref{eq:CFL} for functions that are orthogonal to radial functions. More precisely, we prove that the inequality \eqref{eq:CFL} still holds true with an explicit larger constant for such functions. As an application, we show that $\delta(u)$ provides an upper bound for $\irn |\nabla u_o|^2 dx$ where $u_o(x) = (u(x) -u(-x))/2$ denotes the odd part of $u$. Hence, when $\delta(u)$ is small, the function $u$ is almost even. The second step is to establish a stability estimate of \eqref{eq:CFL} for even functions (see Lemma \ref{leven} below). In order to prove this result, we shall prove an important identity
$$\int_{\R^n} |\Delta u|^2 - (n+2) \irn |\na u|^2 dx + \irn |\na u|^2 |x|^2 dx = \irn \|\na^2 v - x \otimes \na v\|_{HS}^2 e^{-|x|^2}dx$$
%$$\irn \|\na^2 u + \na u \otimes x\|_{HS}^2 dx = \int_{\R^n} |\Delta u|^2 - n \irn |\na u|^2 dx + \irn |\na u|^2 |x|^2 dx$$
with $u = v e^{-|x|^2/2}$, where $\na^2 v$ denotes the { Hessian} matrix of $v$, $\na v\otimes x$ denotes matrix $(\pa_i v x_j)_{i,j=1}^n$ and $\|A\|_{HS}$ denotes the Hilbert-Schmidt norm of an $n\times n$ matrix $A$, i.e., $\|A\|_{HS} = \Big(\text{\rm Tr}(A^tA)\Big)^{1/2}$ with $A^t$ being the transpose of $A$. Using spectral analysis of the Ornstein-Uhlenbeck type operator associated with the Gaussian measure and Hermite polynomials, we arrive at the following estimate for even function $v\in C_0^\infty(\R^n)$
$$\irn \|\na^2 v - x \otimes \na v\|_{HS}^2 e^{-|x|^2}dx \geq \frac{4n}{n+2}\left(\irn |\na ((v-c)e^{-\frac{|x|^2}{2}})|^2dx\right)$$
where
$$
c = \frac{\int_{\R^n} v e^{-|x|^2} dx}{\int_{\R^n} e^{-|x|^2} dx}.
$$
Combining the estimates for the odd and even functions, we obtain  Theorem \ref{Stability}.

The rest of this paper is organized as follows. In Section \S2, we prove the second order CKN inequalities given in Theorem \ref{Main}, Theorem \ref{Mainradial}, and Theorem \ref{MainXiatype}. Section \S3 is devoted to prove the stability version of the second order uncertainty principle of Cazacu, Flynn and Lam given in  Theorem \ref{Stability}.

\section{The second order CKN inequalities: Proof of Theorems \ref{Main}, \ref{Mainradial}, and \ref{MainXiatype}}
In this section, we provide the proof of the second order CKN inequalities in Theorems \ref{Main}, \ref{Mainradial} and  \ref{MainXiatype}. We also show that under the conditions of parameters in these theorems, the obtained inequalities are sharp and we exhibit a class of extremal functions. We begin with the proof of Theorem \ref{Main}.

\begin{proof}[Proof of Theorem \ref{Main}]
For any function $u \in C_0^\infty(\R^n)$, we have
\begin{align}\label{eq:expansionDelta1}
\irn &|\Delta u + \nabla u \cdot x |x|^{2 \alpha }|^2 |x|^{-2\alpha} dx\notag\\
&= \irn |\Delta u|^2 |x|^{-2\al} dx + |\nabla u \cdot x|^2 |x|^{2\alpha} dx + 2 \irn \Delta u \nabla u\cdot x dx.
\end{align}
Using integration by parts, we have
\begin{align*}
\irn \Delta u \nabla u\cdot x dx &= -\irn  \na^2 u (\na u) \cdot x  dx - \irn |\na u|^2dx \\
&= -\frac12 \irn \na (|\na u|^2) \cdot x  dx - \irn |\na u|^2 dx\\
&= \frac{n -2}2 \irn |\na u|^2 dx.
\end{align*}
Inserting this equality in \eqref{eq:expansionDelta1}, we get
\begin{align}\label{eq:eq1}
\irn |\Delta u + \nabla u \cdot x |x|^{2\alpha}|^2 |x|^{-2\alpha} dx&= \irn |\Delta u|^2 |x|^{-2\al} dx + \irn\lt|\nabla u \cdot \frac x{|x|}\rt|^2 |x|^{2+2\alpha} dx \notag\\
&\quad + (n -2) \irn |\na u|^2 dx.
\end{align}

We first consider the case $\alpha\not=-1$. Setting $u = v U_0(x)$ with 
$$U_0(x) = \exp\Big(-\frac{|x|^{2+2\al}}{2 + 2\al}\Big),$$
we have 
$$\na u = \na v U_0(x) + v(x) \na U_0(x),\quad \Delta u = \Delta v U_0 + 2 \na v \na U_0 + v \Delta U_0$$
and
$$\nabla U_0(x) = - x |x|^{2\alpha} e^{-\frac{|x|^{2+2\al}}{2+2\al}}, \, \Delta U_0(x) = -(n +2\al) |x|^{2\alpha} e^{-\frac{|x|^{2+2\al}}{2+2\al}} + |x|^{2 + 4\alpha} e^{-\frac{|x|^{2+2\al}}{2+2\al}}.$$
Plugging this expression into $\Delta u + \na u \cdot x |x|^{2\al}$ and using simple computations imply
$$\Delta u + \na u \cdot x |x|^{2\al} =\Big(\Delta v -\na v \cdot x|x|^{2\al}  - (n+ 2\alpha)v |x|^{2\alpha}\Big) e^{-\frac{|x|^{2+2\al}}{2+2\al}}.$$
Hence, it holds
\begin{align*}
\irn|\Delta u& + \nabla u \cdot x |x|^{2\alpha}|^2 |x|^{-2\alpha} dx\\
&= \irn |\Delta v -\na v \cdot x |x|^{2\alpha}|^2 |x|^{-2\alpha} e^{-\frac{2 |x|^{2+2\al}}{2+2\al}} dx + (n+2\al)^2 \irn v^2 |x|^{2\alpha} e^{-2\frac{|x|^{2+2\al}}{2+2\al}} dx\\
&\quad -2(n+ 2\al) \irn (\Delta v -\na v \cdot x |x|^{2\al}) v e^{-\frac{2 |x|^{2+2\al}}{2+2\al}} dx.
\end{align*}
We next compute the last integral in the right-hand side of the preceding equality. Notice that the function $v = u e^{\frac{|x|^{2+2\al}}{2+2\al}}$ is not $C^2$ at origin in general. So we can not use integration by parts directly. To overcome this difficulty, we first notice that under the assumption $u \in C_0^\infty(\R^n)$, $n+ 2\al >0$ and $n + 2+ 4\al >0$, we have
$$v \Delta v e^{-\frac{2|x|^{2+2\al}}{2+2\al}} = u(\Delta u + 2\na u \cdot x |x|^{2\alpha} + (n+2\al)|x|^{2\al} u + |x|^{2+4\al} u)\in L^1(\R^n)$$ 
and
$$v\na v\cdot x |x|^{2\al} e^{-\frac{2|x|^{2+2\al}}{2+2\al}} = u (\na u \cdot x |x|^{2\al} -u |x|^{2+ 4\alpha}) \in L^1(\R^n).$$
Therefore it holds
\begin{align*}
\irn (\Delta v -\na v \cdot x |x|^{2\al}) v e^{-\frac{2 |x|^{2+2\al}}{2+2\al}} dx &= \lim_{\ep\to 0^+} \int_{B_\ep^c}(\Delta v -\na v \cdot x |x|^{2\al}) v e^{-\frac{2 |x|^{2+2\al}}{2+2\al}} dx,
\end{align*}
where $B_\ep^c = \{x\in \R^n \, :\, |x| \geq \epsilon\}$. Using integration by parts we have
\begin{align*}
&\int_{B_\ep^c}(\Delta v -\na v \cdot x |x|^{2\al}) v e^{-\frac{2 |x|^{2+2\al}}{2+2\al}} dx \\
&= \int_{B_\ep^c}\text{\rm div}\big(\na v e^{-\frac{ |x|^{2+2\al}}{2+2\al}}\big) v e^{-\frac{ |x|^{2+2\al}}{2+2\al}} dx\\
&=-\int_{B_\ep^c}\na v e^{-\frac{ |x|^{2+2\al}}{2+2\al}} \na (v e^{-\frac{ |x|^{2+2\al}}{2+2\al}}) dx + \int_{\{|x|=\ep\}}\na v \cdot \frac x{|x|} v e^{-2\frac{\ep^{2+2\al}}{2+2\al}} ds\\
&= -\int_{B_\ep^c}|\na v|^2 e^{-\frac{ |2x|^{2+2\al}}{2+2\al}}dx + \frac12 \int_{B_\ep^c}\na v^2 \cdot x |x|^{2\al} e^{-\frac{ 2|x|^{2+2\al}}{2+2\al}}dx + \int_{\{|x|=\ep\}}\na v \cdot \frac x{|x|} v e^{-2\frac{\ep^{2+2\al}}{2+2\al}} ds  \\
&=   -\int_{B_\ep^c}|\na v|^2 e^{-\frac{ 2|x|^{2+2\al}}{2+2\al}}dx - \frac{n+2\al}2 \int_{B_\ep^c} v^2 |x|^{2\al} e^{-\frac{ 2|x|^{2+2\al}}{2+2\al}}dx + \int_{B_\ep^c} v^2 |x|^{2+ 4\al} e^{-\frac{ 2|x|^{2+2\al}}{2+2\al}}dx \\
&\qquad +\ep^{1+2\al} e^{-2\frac{\ep^{2+2\al}}{2+2\al}} \int_{\{|x|=\ep\}} v^2 ds + \int_{\{|x|=\ep\}}\na v \cdot \frac x{|x|} v e^{-2\frac{\ep^{2+2\al}}{2+2\al}} ds. 
\end{align*} 
Letting $\ep \to 0^+$ and using the assumptions $n + 2\al >0$ and $n + 2+ 4\al >0$, we obtain
\begin{align*}
\lim_{\ep \to 0^+} \int_{B_\ep^c}(\Delta v &-\na v \cdot x |x|^{2\al}) v e^{-\frac{2 |x|^{2+2\al}}{2+2\al}} dx \\
&= -\irn|\na v|^2 e^{-\frac{ 2|x|^{2+2\al}}{2+2\al}}dx - \frac{n+2\al}2 \irn v^2 |x|^{2\al} e^{-\frac{ 2|x|^{2+2\al}}{2+2\al}}dx\\
&\qquad + \irn v^2 |x|^{2+ 4\al} e^{-\frac{ 2|x|^{2+2\al}}{2+2\al}}dx.
\end{align*}
Consequently, we arrive at
\begin{align}\label{eq:toetoet}
\irn|&\Delta u + \nabla u \cdot x |x|^{2\alpha}|^2 |x|^{-2\alpha} dx\notag\\
&= \irn |\Delta v -\na v \cdot x |x|^{2\alpha}|^2 |x|^{-2\alpha} e^{-\frac{2 |x|^{2+2\al}}{2+2\al}} dx +2(n+2\al) \irn|\na v|^2 e^{-\frac{ 2|x|^{2+2\al}}{2+2\al}}dx \notag\\
&\quad +2 (n+2\al)^2 \irn v^2 |x|^{2\alpha} e^{-\frac{2|x|^{2+2\al}}{2+2\al}} dx - 2(n+2\alpha)\irn v^2 |x|^{2+ 4\al} e^{-\frac{ 2|x|^{2+2\al}}{2+2\al}}dx.
\end{align}
Again by using integration by parts on $B_\ep^c$ and letting $\ep \to 0^+$, we have
\begin{align}\label{eq:eq3}
\irn |\na u|^2 dx &= \irn |\na v|^2 U_0^2 dx + \irn v^2 |\na U_0|^2  dx + \int_{\R^n} \na v^2 \cdot U_0 \na U_0  dx \notag\\
&= \irn |\na v|^2 U_0^2dx  - \irn v^2 U_0 \Delta U_0  dx\notag\\
&=\irn |\na v|^2 e^{-2\frac{|x|^{2+ 2\al}}{2+2\al}} dx + (n+ 2\al) \irn v^2 |x|^{2\al} e^{-2\frac{|x|^{2+ 2\al}}{2+2\al}}dx\notag\\
&\quad - \irn v^2 |x|^{2 +4\al} e^{-2\frac{|x|^{2+ 2\al}}{2+2\al}} dx.
\end{align}
Combining \eqref{eq:eq1}, \eqref{eq:toetoet} and \eqref{eq:eq3}, we obtain
\begin{align*}
\irn& |\Delta u|^2 |x|^{-2\al} dx + \irn\lt|\nabla u \cdot \frac x{|x|}\rt|^2 |x|^{2+2\alpha} dx + (n -2) \irn |\na u|^2 dx\\
&=\irn |\Delta v -\na v \cdot x |x|^{2\alpha}|^2 |x|^{-2\alpha} e^{-\frac{2 |x|^{2+2\al}}{2+2\al}} dx + 2(n+2\al) \irn |\na u|^2 dx
\end{align*}
which is equivalent to
\begin{align}\label{eq:general}
\irn |\Delta u|^2& |x|^{-2\al} dx + \irn\lt|\nabla u \cdot \frac x{|x|}\rt|^2 |x|^{2+2\alpha} dx \notag\\
& = \irn |\Delta v -\na v \cdot x |x|^{2\alpha}|^2 |x|^{-2\alpha} e^{-\frac{2 |x|^{2+2\al}}{2+2\al}} dx + (n+4\al+2) \irn |\na u|^2 dx.
\end{align}
By density argument, \eqref{eq:general} still holds for any functions $u \in H_{\al,-2-2\al}^2(\R^n)$. It follows from \eqref{eq:general} that
$$\irn |\Delta u|^2 |x|^{-2\al} dx + \irn\lt|\nabla u \cdot \frac x{|x|}\rt|^2 |x|^{2+2\alpha} dx \geq (n+4\al+2) \irn |\na u|^2 dx$$
for any $u \in H_{\al, -2-2\beta}^2(\R^n)$. Replacing $u$ by function $u_\lam(x) = \lam^{\frac{n-2}2} u(\lam x)$ with $\lam >0$, we get
$$\lam^{2 + 2\al}\irn |\Delta u|^2 |x|^{-2\al} dx + \lam^{-2 -2\al} \irn\lt|\nabla u \cdot \frac x{|x|}\rt|^2 |x|^{2+2\alpha} dx \geq (n+4\al+2) \irn |\na u|^2 dx.$$
The left-hand side of the preceding inequality is minimized by 
$$\lam_0 = \lt(\frac{\irn |\Delta u|^2 |x|^{-2\al} dx}{ \irn\lt|\nabla u \cdot \frac x{|x|}\rt|^2 |x|^{2+2\alpha} dx}\rt)^{\frac1{4+4\al}}.$$
Hence, by taking $\lam =\lam_0$, we obtain \eqref{eq:CKNalpha} for $\alpha \not= -1$. The case $\alpha =-1$ follows by letting $\alpha \to -1$.

It remains to check the sharpness of \eqref{eq:CKNalpha} when $1 + \al >0$. Taking $u = U_0$ implies $v\equiv 1$. Hence, \eqref{eq:general} becomes
$$\irn |\Delta U_0|^2 |x|^{-2\al} dx + \irn\lt|\nabla U_0\cdot \frac x{|x|}\rt|^2 |x|^{2+2\alpha} dx = (n+4\al+2) \irn |\na U_0|^2 dx.$$
Furthermore, by the direct computations and integration by parts, we have
\begin{align*}
\irn |\Delta U_0|^2 |x|^{-2\al} dx& = (n+2\al)^2 \irn |x|^{2\al} U_0^2dx -2(n+2\al)\irn |x|^{2+ 4\al} U_0^2dx \\
&\qquad + \irn |x|^{4+ 6\al} U_0^2 dx\\
&=(n+2\al)^2 \irn |x|^{2\al} U_0^2 dx  + (n+2\al) \irn \na U_0^2 \cdot x |x|^{2\al} dx \\
&\qquad + \irn\lt|\nabla U_0\cdot \frac x{|x|}\rt|^2 |x|^{2+2\alpha} dx\\
&=\irn\lt|\nabla U_0\cdot \frac x{|x|}\rt|^2 |x|^{2+2\alpha} dx.
\end{align*}
This implies that the equality occurs in \eqref{eq:CKNalpha} with $u= U_0$. Hence, the inequality \eqref{eq:CKNalpha} is sharp and $U_0$ is an extremal function. This completes the proof of Theorem \ref{Main}.
\end{proof}

We next prove Theorem \ref{Mainradial}.
\begin{proof}[Proof of Theorem \ref{Mainradial}]
By density argument, it is enough to prove \eqref{eq:2CKN} for radial functions $u \in C_0^\infty(\R^n)$. Let $u \in C^\infty_0(\R^n)$ be a radial function, by using integration by parts, we have
\begin{align*}
\int_{\R^n} \frac{|\na u|^t}{|x|^{t\ga}} dx &= |S^{n-1}| \int_0^\infty |u'|^t r^{n-t\ga -1}dr\\
&= \frac{1}{n -t\ga} |S^{n-1}| \int_0^\infty |u'|^t (r^{n-t\ga})' dr\\
&= -\frac t{n-t\ga} |S^{n-1}| \int_0^\infty |u'|^{t-2} u' u'' r^{n-t\ga} dr\\
&= -\frac t{n-t\ga} |S^{n-1}| \int_0^\infty |u'|^{t-2} u'\Big(u'' + \frac{n-1}{r} u' - \frac{n+2\al}r u'\Big) r^{n-t\ga} dr\\
&\quad -\frac{t(1+2\al)}{n-t\ga} |S^{n-1}| \int_0^\infty |u'|^t r^{n-t\ga -1}dr.
\end{align*}
This gives
\begin{align}\label{eq:new}
 &\frac{n+t(1+ 2\al -\ga)}t\int_{\R^n} \frac{|\na u|^t}{|x|^{t\ga}} dx \notag\\
&\qquad\qquad= -|S^{n-1}| \int_0^\infty |u'|^{t-2} u'\Big(u'' + \frac{n-1}{r} u' - \frac{n+2\al}r u'\Big) r^{n-2\ga} dr\notag\\
&\qquad\qquad= -|S^{n-1}| \int_0^\infty |u'|^{t-2} u' r^{-\frac\be 2} \Big(u'' + \frac{n-1}{r} u' - \frac{n+2\al}r u'\Big) r^{-\alpha} r^{n-1} dr,
\end{align}
here we use \eqref{eq:balance}. By density argument, \eqref{eq:new} still holds for radial function $u \in H_{\al,\beta}^2(\R^n)$. Using H\"older inequality, we arrive at 
\begin{align}\label{eq:trunggian}
& \left|\frac{n+t(1+ 2\al -\ga)}t\right|\int_{\R^n} \frac{|\na u|^t}{|x|^{t\ga}} dx \notag\\
&\leq \lt(|S^{n-1}| \int_0^\infty \Big(u'' + \frac{n-1}{r} u' - \frac{n+2\al}r u'\Big)^2r^{n-2\al -1}dr\rt)^{\frac12}\notag\\
&\qquad \times \lt(|S^{n-1}| \int_0^\infty |u'|^{2(t-1)} r^{n-\beta-1}dr\rt)^{\frac12} \notag\\
&= \lt(|S^{n-1}| \int_0^\infty \Big(u'' + \frac{n-1}{r} u' - \frac{n+2\al}r u'\Big)^2r^{n-2\al -1}dr\rt)^{\frac12}\lt(\irn \frac{|\na u|^{2(t-1)}}{|x|^\be} dx\rt)^{\frac12} .
\end{align}
Furthermore, using integration by parts, we have
\begin{align*}
&\int_0^\infty \Big(u'' + \frac{n-1}{r} u' - \frac{n+2\al}r u'\Big)^2r^{n-2\al -1}dr\\
&= \int_0^\infty (\Delta u(r))^2 r^{n-2\al -1} dr + (n+2\al)^2 \int_0^\infty (u')^2 r^{n-2\al -3} dr\\
&\quad -2(n+2\alpha) \int_0^\infty\Big( u'' + \frac{n-1}{r} u'\Big) \frac{u'}r r^{n-2\al -1} dr\\
&= \int_0^\infty (\Delta u(r))^2 r^{n-2\al -1} dr - (n+2\al)(n-2\al -2)\int_0^\infty (u')^2 r^{n-2\al -3} dr\\
&\quad -(n+2\al) \int_0^\infty ((u')^2)' r^{n-2\al -2} dr\\
&= \int_0^\infty (\Delta u(r))^2 r^{n-2\al -1} dr.
\end{align*}
Consequently, it holds
$$|S^{n-1}| \int_0^\infty \Big(u'' + \frac{n-1}{r} u' - \frac{n+2\al}r u'\Big)^2r^{n-2\al -1}dr = \irn \frac{(\Delta u)^2}{|x|^{2\al}}dx.$$
Inserting this equality into \eqref{eq:trunggian}, we obtain \eqref{eq:2CKN}.

Suppose that a nonzero radial function $u\in H_{\al,\beta}^2(\R^n)$ is an extremal function for \eqref{eq:2CKN}. Notice that under the conditions \eqref{eq:welldefineu} and \eqref{eq:Deltau}, we have
$$n+t(1+2\al -\ga) = n+ 2\al + (1+2\al)(t-2) + 1 + \al -\frac\beta 2 >0.$$
Hence, the equation holds when applying the H\"older inequality to \eqref{eq:new} if and only if  
$$u'' + \frac{n-1}{r} u' - \frac{n+2\al}r u' = -\lambda |u'|^{t-2} u' r^{\al -\frac \beta2}$$
for some $\lambda >0$, which is equivalent to
$$u'' -\frac{1+2\al}{r} u' + \lambda |u'|^{t-2} u' r^{\al -\frac \beta2} = 0.$$
Denote $u' = r^{1+ 2\al} w$, then $w$ satisfies the equation
$$w' + \lam r^{(1+2\al)(t-2) + \al -\frac \beta2}|w|^{t-2} w =0.$$
We have following two cases:

{\bf Case 1: $t=2$.} In this case, we have $w' + \lam r^{\al -\frac\be 2} w =0$ which implies $w(r) = c \exp(-\lam r^{1+ \al -\frac\be2}/(1 + \al -\frac\be 2))$ for some $c \in \R$. Hence, 
$$u'(r) = c r^{1 + 2\al} \exp\Big(-\lam \frac{r^{1 + \al -\be/2}}{1 + \al -\be/2}\Big)$$
and
$$u(x) = c \int_{|x|}^\infty r^{1 + 2\al} \exp\Big(-\lam\frac{r^{1 + \al -\be/2}}{1 + \al -\be/2}\Big)dr.$$

{\bf Case 2: $t > 2$.} In this case, we have $(|w|^{2-t})' = \lam(t-2)\lam r^{(1+2\al)(t-2) + \al -\frac \beta2}$ which implies
$$|w(r)| = \Big(c + \lam (t-2)\frac{r^{(1+2\al)(t-2) + 1 + \al -\frac\be2}}{(1+2\al)(t-2) + 1 + \al -\frac\be2}\Big)^{\frac1{2-t}},$$
for some $c >0$, here we use \eqref{eq:welldefineu}. From this expression, up to a multiplicative constant $1$ or $-1$, we can assume that
$$w(r) = \Big(c + \lam (t-2)\frac{r^{(1+2\al)(t-2) + 1 + \al -\frac\be2}}{(1+2\al)(t-2) + 1 + \al -\frac\be2}\Big)^{\frac1{2-t}}.$$
Therefore, the extremal function has the form
$$
u(x) = \int_{|x|}^\infty r^{1+ 2\al} \Big(c + \lam (t-2)\frac{r^{(1+2\al)(t-2) + 1 + \al -\frac\be2}}{(1+2\al)(t-2) + 1 + \al -\frac\be2}\Big)^{\frac1{2-t}} dr
$$
as desired.
\end{proof}

We finish this section by proving Theorem \ref{MainXiatype}. Our proof follows the approach of Cazacu, Flynn and Lam to prove \eqref{eq:CFL} by using the decomposition of $u$ into spherical harmonics. It is well known that the technique of decomposing a function into spherical harmonics is a very useful method to prove the Hardy-Rellich type inequalities (see \cite{Ter,Ghou,Caz,NguyenProcA1,VZ} and references therein). Let us recall some facts on this method which we borrow from \cite[Section $2.2$]{Ter}. A function $f \in C_0^\infty(\R^n)$ can be decomposed into spherical harmonics as
\begin{equation}\label{eq:Spher}
f(x) = \sum_{k=0}^\infty f_k(r) \phi_k(\omega),\quad x = r \omega, |x| =r, \omega \in S^{n-1}
\end{equation}
where $\phi_k$ are the orthogonal eigenfunctions of the Laplace-Beltrami operator on $S^{n-1}$ with the corresponding eigenvalue $c_k =k(n+k-2)$. Notice that $\phi_0\equiv 1$, $\phi_k$ is restriction of the $k$ order homogeneous, harmonic polynomials in $\R^n$ to $S^{n-1}$ and $f_k(r) =\frac{1}{|S^{n-1}|}\int_{S^{n-1}} f \phi_k ds$ with $k\geq 0$. Hence $f_k \in C_0^\infty(\R^n)$ with $f_k(r) = O(r^k)$ as $r\to 0$.

We have
$$\Delta f(x) = \sum_{k=0}^\infty \Big(f_k''(r) + \frac{n-1}r f_k'(r) -c_k \frac{f_k(r)}{r^2}\Big)\phi_k(\omega)$$
and
$$|\na f(x)|^2 = \sum_{k=0}^\infty \Big(|\nabla f_k|^2 \phi_k^2 + \frac{f_k^2}{r^2} |\na_{S^{n-1}} \phi_k|^2\Big).$$
Following Cazacu, Flynn and Lam, we set $f_k(r) = r^k g_k(r)$. By the simple computations, we have
\begin{align*}
f_k''(r) &+ \frac{n-1}r f_k'(r) -c_k \frac{f_k(r)}{r^2}\\
& = r^k g_k'' + 2k r^{k-1} g_k' + k(k-1) r^{k-2} g_k + (n-1)r^{k-1} g_k' + k(n-1) r^{k-2} g_k -c_k r^{k-2} g_k\\
&=r^k g_k'' + (n + 2k -1) r^{k-1} g_k'
\end{align*}
and 
$$|\na f_k|^2 = k^2 r^{2(k-1)} g_k^2 + r^{2k} (g_k')^2 + 2k r^{2k-1} g_k g_k'.$$
Therefore, using integration by parts and the definitions \eqref{eq:Ak}, \eqref{eq:Bk} and \eqref{eq:Ck}, we get
\begin{align}\label{eq:Deltaalpha}
\int_{\R^n} \frac{|\Delta f|^2}{|x|^{2\al}} dx &= \sum_{k=0}^\infty \int_0^\infty \Big(f_k''(r) + \frac{n-1}r f_k'(r) -c_k \frac{f_k(r)}{r^2}\Big)^2 r^{n-2\al-1} dr\notag\\
&=\sum_{k=0}^\infty \int_0^\infty \Big(g_k''(r) + \frac{n+2k-1}r g_k'(r)\Big)^2 r^{n-2\al+ 2k- 1} dr\notag\\
&= \sum_{k=0}^\infty \Big(\int_0^\infty (g_k'')^2 r^{n+2k -2\al -1} dr + (n+2k-1)^2 \int_0^\infty (g'_k)^2 r^{n +2k -2\al -3} dr\notag\\
&\qquad\qquad\qquad \qquad + 2(n+2k-1) \int_0^\infty g_k g_k' r^{n+2k -2\al -2} dr\Big)\notag\\
&= \sum_{k=0}^\infty \Big(\int_0^\infty (g_k'')^2 r^{n+2k -2\al -1} dr \notag\\
&\qquad\qquad\qquad\quad \quad\quad + (1 + 2\al) (n+2k-1) \int_0^\infty (g'_k)^2 r^{n +2k -2\al -3} dr\Big)\notag\\
&= \sum_{k =0}^\infty A_{n,\al,k}(g_k),
\end{align}
\begin{align}\label{eq:nablabeta}
\irn \frac{|\nabla f|^2}{|x|^\beta}dx &= \sum_{k=0}^\infty  \int_0^\infty \Big((f_k')^2 + c_k \frac{f_k^2}{r^2}\Big) r^{n-\beta -1} dr\notag\\
&=  \sum_{k=0}^\infty \Big(\int_0^\infty (g_k')^2 r^{n+2k -\beta -1}dr + 2k\int_0^\infty g_k g_k' r^{n+2k -\beta -2} dr \notag\\
&\qquad\qquad\qquad\quad + (c_k + k^2) \int_0^\infty g_k^2 r^{n+2k -\beta -3} dr\notag\\
&= \sum_{k=0}^\infty \Big( \int_0^\infty (g_k')^2 r^{n+2k -\beta -1}dr + \beta k \int_0^\infty g_k^2 r^{n+2k -\beta -3} dr\Big)\notag\\
&= \sum_{k=0}^\infty B_{n,\beta,k}(g_k)
\end{align}
and
\begin{align}\label{eq:nablagamma}
\irn \frac{|\nabla f|^2}{|x|^{2\ga}}dx &=  \sum_{k=0}^\infty \Big( \int_0^\infty (g_k')^2 r^{n+2k -2\ga -1}dr + 2\gamma k \int_0^\infty g_k^2 r^{n+2k -2\ga -3} dr\Big)\notag\\
& =\sum_{k=0}^\infty C_{n,\ga,k}(g_k).
\end{align}
With these preparations, we are ready to prove Theorem \ref{MainXiatype}

\begin{proof}[Proof of Theorem \ref{MainXiatype}]
It follows from \eqref{eq:Deltaalpha}, \eqref{eq:nablabeta} and \eqref{eq:nablagamma} that 
$$\Big(\irn \frac{|\Delta f|^2}{|x|^{2\al}} dx\Big)\Big(\irn \frac{|\na f|^2}{|x|^\beta} dx\Big) = |S^{n-1}|^2 \Big(\sum_{k=0}^\infty A_{n,\al,k}(g_k)\Big) \Big(\sum_{k=0}^\infty B_{n,\be,k}(g_k)\Big)$$
and 
$$
\irn \frac{|\nabla f|^2}{|x|^{2\ga}}dx = |S^{n-1}|\sum_{k=0}^\infty C_{n,\ga,k}(g_k).
$$
By Minkowski inequality and the definition \eqref{eq:sharpAk} of $A_k(n,\al,\beta)$, we have
\begin{align*}
\Big(\sum_{k=0}^\infty  A_{n,\al,k}(g_k)\Big) \Big(\sum_{k=0}^\infty B_{n,\be,k}(g_k)\Big) &\geq \Big(\sum_{k=0}^\infty \sqrt{A_{n,\al,k}(g_k) B_{n,\be,k}(g_k)}\Big)^2 \\
&\geq \inf_{k\in \N} A_k(n,\al,\beta) \Big(\sum_{k=0}^\infty C_{n,\ga,k}(g_k)\Big)^2,
\end{align*}
which yields
$$
\Big(\irn \frac{|\Delta f|^2}{|x|^{2\al}} dx\Big)\Big(\irn \frac{|\na f|^2}{|x|^\beta} dx\Big) \geq \inf_{k\in \N} A_k(n,\al,\beta)\Big(\irn \frac{|\nabla f|^2}{|x|^{2\ga}}dx\Big)^2.
$$

Furthermore, by using one dimensional Hardy inequality, we have
$$\int_0^\infty (g')^2 r^{n+2k -\beta -1}dr \geq \Big(\frac{n+2k -\beta -2}2\Big)^2 \int_0^\infty g^2 r^{n+2k -\beta -3} dr,$$
and 
$$ \int_0^\infty (g')^2 r^{n+2k -2\ga -1}dr \geq\Big(\frac{n+2k -2\ga -2} 2\Big)^2 \int_0^\infty g^2 r^{n+2k -2\ga -3}dr$$
which imply
$$B_{n,\be,k}(g) \geq \Big(1 + \min\Big\{0, \frac{4\be k}{(n+ 2k -\beta -2)^2}\Big\}\Big) \int_0^\infty (g')^2 r^{n+2k -\beta -1}dr$$
and 
$$C_{n,\gamma,k}(g) \leq \Big(1 + \max\Big\{0, \frac{8\ga k}{(n+ 2k -2\ga -2)^2}\Big\}\Big) \int_0^\infty (g')^2 r^{n+2k -2\ga -1}dr.$$
Notice that
$$1 + \min\Big\{0, \frac{4\be k}{(n+ 2k -\beta -2)^2}\Big\} >0$$
and
$$A_{n,\al,k}(g) = \int_0^\infty \Big(g''(r) + \frac{n+2k-1}r g'(r)\Big)^2 r^{n+2k -2\al -1} dr.$$
Hence, applying the inequality \eqref{eq:2CKN} for radial functions in dimension $n+2k$ we get 
\begin{equation}\label{eq:Akestimate}
A_k(n,\al,\beta) \geq \frac{1 + \min\Big\{0, \frac{4\be k}{(n+ 2k -\beta -2)^2}\Big\}}{1 + \max\Big\{0, \frac{8\ga k}{(n+ 2k -2\ga -2)^2}\Big\}} \Big(\frac{n+ 2 k + 4\al -2\ga +2}{2}\Big)^2.
\end{equation}
This gives
$$\lim_{k\to \infty} A_k(n,\al,\beta) = \infty$$
and hence
$$\inf_k A_k(n,\al,\beta) = \min_k A_k(n,\al,\beta).$$
Consequently, we get
$$\Big(\irn \frac{|\Delta f|^2}{|x|^{2\al}} dx\Big)\Big(\irn \frac{|\na f|^2}{|x|^\beta} dx\Big) \geq \min_{k} A_k(n,\al,\beta)\Big(\irn \frac{|\nabla f|^2}{|x|^{2\ga}}dx\Big)^2$$
as wanted \eqref{eq:constant}.

It is easy to see that the constant $\min_k A_k(n,\al,\beta)$ is sharp in \eqref{eq:constant}. Indeed, there exists $k_0$ and a sequence of function $h_k$ such that $\min_k A_k(n,\al,\beta) = A_{k_0}(n,\al,\beta)$ and
$$\lim_{k\to \infty} \frac{A_{n,\al,k_0}(h_k)B_{n,\beta,k_0}(h_k)}{C_{n,\ga,k_0}(h_k)^2} = A_{k_0}(n,\al,\beta).$$
Testing \eqref{eq:constant} by functions $h_k \phi_{k_0}$ implies the sharpness of $\min_k A_k(n,\al,\be)$.

Finally, the estimate \eqref{eq:estimateforsharp} immediately follows from \eqref{eq:Akestimate}.

%It remains to estimate $A_k(n,\al,\be)$ and find conditions such that
%$$\min_{k} A_k(n,\al,\beta) =\Big(\frac{n+ 4\al -2\ga -2}2\Big)^2 = A_0(n,\al,\beta).$$

\end{proof}

\section{The proof of Theorem \ref{Stability}}
In this section, we prove the stability version of the second order uncertainty principle \eqref{eq:CFL} due to Cazacu, Flynn and Lam given in Theorem \ref{Stability}. We divided the proof into two parts. In the first part, we show that $\de(u)$ gives an upper bound for the odd part of the function $u$, hereafter for a function $u$ we call the function $u_o(x) = (u(x) - u(-x))/2$ as its odd part. This task is done by establishing an improvement of \eqref{eq:CFL} on the odd functions (even more general, on functions that are orthogonal to all radial functions). Consequently, when $\de(u)$ is small, $u$ is almost an even function. In the second part, we prove the stability estimate \eqref{eq:Stabestimate} for even functions. This is done by using spectral analysis of the Orstein-Uhlenbeck operators on the Gaussian space. Combining these estimates, we prove theorem \ref{Stability}.  %We start by preparing some main ingredients in our proof.Again, by density argument, it is enough to prove these theorems for functions $u \in C_0^\infty(\R^n)$.

\subsection{Estimate for the odd functions}
We start this subsection by proving an improvement of \eqref{eq:CFL} for compactly supported smooth functions that are orthogonal to all radial functions. A function $u$ is called to be orthogonal to all radial functions if 
\[
\int_{\R^n} u(x) \varphi(x) dx = 0
\]
for any radial function $\varphi$. More precisely, we will prove the following result.
\begin{theorem}\label{improvedCFL}
Given $n \geq 2$. For any function $u\in C_0^\infty(\R^n)$ that is orthogonal to all radial functions, it holds
\begin{equation}\label{eq:improvedCFL}
\Big(\irn |\Delta u|^2 dx\Big)^{\frac12} \Big( \irn |x|^2 |\na u|^2 dx\Big)^{\frac 12} \geq C_1(n) \irn |\na u|^2 dx
\end{equation}
where
\[
C_1(n) = \frac{n+ 4}2 \Big(1 - \frac 8{(n+2)^2}\Big)^{\frac12}.
\]
\end{theorem} 
%Evidently, we have $C_1(n) > (n+2)/2$ ....
\begin{proof}
We follows the proof of \eqref{eq:CFL} given in \cite{CFL2020} by using the spherical harmonic decomposition \eqref{eq:Spher} (see also the proof of Theorem \ref{MainXiatype}). Since $u$ is orthogonal to all radial function, then 
\[
u(x) = \sum_{k =1}^\infty u_k(r) \phi_k(\omega), \quad r =|x|,\, x = r \omega.
\]
Let $u_k(r) = r^k v_k(r)$, then we have from \eqref{eq:Deltaalpha}, \eqref{eq:nablabeta} and \eqref{eq:nablagamma} that
\[
\int_{\R^n} (\Delta u)^2 dx = \sum_{k =1}^\infty A_{n,0,k}(v_k) ,
\]
\[
\int_{\R^n} |x|^2 |\na u|^2 dx = \sum_{k =1}^\infty B_{n,-2,k}(v_k),
\]
and
\[
\int_{\R^n}|\na u|^2 dx = \sum_{k =1}^\infty C_{n,0,k}(v_k).
\]
It was proved in \cite[Proof of Theorem $2.1$, Case $N\geq 2$]{CFL2020} that
\begin{align*}
A_{n,0,k}(v_k) B_{n,-2,k}(v_k) & =\Big(\int_0^\infty r^{n+2k-1} (v_k''(r))^2 dr + (n+ 2k -1) \int_0^\infty r^{n+ 2k -3} (v_k'(r))^2 dr\Big)\\
&\quad\quad\quad \quad \times \Big(\int_0^\infty r^{n+2k+1} (v_k'(r))^2 dr -2k \int_0^\infty r^{n+ 2k -1} v_k(r)^2 dr\Big)\\
& \geq \Big(1 - \frac{8k}{(n+ 2k)^2} \Big) \frac{(n+ 2k +2)^2}4 \int_0^\infty r^{n+2k-1} (v_k'(r))^2 dr\\
&= \Big(1 - \frac{8k}{(n+ 2k)^2} \Big) \frac{(n+ 2k +2)^2}4 C_{n,0,k}(v_k)^2.
\end{align*}
Using Cauchy-Schwarz inequality, we have
\begin{align*}
\Big(\irn |\Delta u|^2 dx\Big)^{\frac12} \Big( \irn |x|^2 |\na u|^2 dx\Big)^{\frac 12} &= \Big(\sum_{k=1}^\infty A_{n,0,k}(v_k)\Big)^{\frac12} \Big(\sum_{k=1}^\infty B_{n,-2,k}(v_k)\Big)^{\frac12}\\
&\geq \sum_{k=1}^\infty \sqrt{A_{n,0,k}(v_k) B_{n,-2,k}(v_k)} \\
&\geq \inf_{k\geq 1} \Big(1 - \frac{8k}{(n+ 2k)^2} \Big)^{\frac12} \frac{n+ 2k +2}2 \sum_{k=1}^\infty C_{n,0,k}(v_k) \\
&= \inf_{k\geq 1} \Big(1 - \frac{8k}{(n+ 2k)^2} \Big)^{\frac12} \frac{n+ 2k +2}2 \int_{\R^n} |\na u|^2 dx. 
\end{align*}
Since $n\geq 2$, it was show in the proof of Lemma $3.3$ in \cite{CFL2020} that
\[
\inf_{k\geq 1} \Big(1 - \frac{8k}{(n+ 2k)^2} \Big)^{\frac12} \frac{n+ 2k +2}2 = \frac{n+ 4}2 \Big(1 - \frac 8{(n+2)^2}\Big)^{\frac12}.
\]
This completes the proof of this theorem.
\end{proof}

It is easy to see that $\sqrt{1 -2 t^2} \geq 1 - (2 -\sqrt{2}) t$ for any $0 \leq t \leq \frac 12$, and hence
\[
(1 + t) \sqrt{1 -2 t^2} \geq (1 + t) (1 - (2 -\sqrt{2}) t) = 1 + t( \sqrt 2 -1 - (2 -\sqrt 2) t) \geq 1 + \frac t{10}
\]
for any $0\leq t \leq \frac 12$. Using this estimate, we get
\begin{equation}\label{eq:estConst}
C_1(n) = \frac{n+ 4}2 \Big(1 - \frac 8{(n+2)^2}\Big)^{\frac12} \geq \frac{n+2}2 + \frac 1{10}.
\end{equation}
Since $C_1(n) > (n+2)/2$, then the inequality \eqref{eq:improvedCFL} provides an improvement of \eqref{eq:CFL} for functions that are orthogonal to radial functions.

Let $u\in C_0^\infty(\R^n)$ be a function such that $\int_{\R^n} |\na u|^2 dx = 1$, we define $u_o(x) = (u(x) - u(-x))/ 2$ and $u_e(x) = (u(x) + u(-x))/2$ the odd part and even part of $u$ respectively. Note that $u_o$ is an odd function, $u_e$ is an even function and $u = u_o + u_e$. It is easy to check that
\[
\int_{\R^n} (\Delta u)^2 dx = \int_{\R^n} (\Delta u_o)^2 dx + \int_{\R^n} (\Delta u_e)^2 dx,
\]
\[
\int_{\R^n} |\nabla u|^2 |x|^2 dx = \int_{\R^n} |\nabla u_o|^2 |x|^2 dx + \int_{\R^n} |\nabla u_e|^2 |x|^2 dx
\]
and
\[
1 = \int_{\R^n} |\nabla u|^2 dx = \int_{\R^n} |\nabla u_o|^2  dx + \int_{\R^n} |\nabla u_e|^2  dx.
\]
Using the Cauchy-Schwarz inequality, we have
\begin{align}\label{eq:estodd}
\delta (u) &:=\frac2{n+2} \lt(\Big(\irn |\Delta u|^2 dx\Big)^{\frac12} \Big(\irn |x|^2 |\na u|^2 dx\Big)^{\frac12} - \frac {n+2}2\rt)\notag\\
&=\frac2{n+2} \Bigg(\Big(\irn |\Delta u_o|^2 dx + \irn |\Delta u_e|^2 dx\Big)^{\frac12} \Big(\irn |x|^2 |\na u_o|^2 dx + \irn |x|^2 |\na u_o|^2 dx\Big)^{\frac12} \notag\\
&\qquad\qquad \qquad - \frac {n+2}2 \irn  |\na u_o|^2 dx - \frac{n+2}2 \irn |\na u_e|^2 dx\Bigg)\notag\\
&\geq \frac2{n+2} \Bigg(\Big(\irn |\Delta u_o|^2 dx\Big)^{\frac12} \Big(\irn |x|^2 |\na u_o|^2 dx\Big)^{\frac12} + \Big(\irn |\Delta u_e|^2 dx\Big)^{\frac12} \Big(\irn |x|^2 |\na u_e|^2 dx\Big)^{\frac12}\notag\\
&\qquad\qquad\qquad - \frac {n+2}2 \irn  |\na u_o|^2 dx - \frac{n+2}2 \irn |\na u_e|^2 dx\Bigg)\notag\\
&\geq \frac{1}{5(n+2)} \int_{\R^n} |\na u_o|^2 dx + \delta(u_e) \int_{\R^n} |\na u_e|^2 dx.
\end{align}
Here we used \eqref{eq:estConst} for the last inequality. As consequences of \eqref{eq:estodd}, we have  the following results

\begin{theorem}\label{oddfunct}
Given $n \geq 2$. For any function $u \in C_0^\infty(\R^n)$ with $\int_{\R^n} |\na u|^2 dx =1$ and $\delta(u) \leq \frac{1}{10 (n+2)}$, it holds
\begin{equation}\label{eq:oddpart}
\int_{\R^n} |\na u_o|^2 dx \leq 5(n+2) \delta(u),
\end{equation}
and
\begin{equation}\label{eq:evenpart}
\delta(u_e) \leq 2 \delta(u).
\end{equation}
\end{theorem}
\begin{proof}
Since $\delta(u_e) \geq 0$, we have from \eqref{eq:estodd}
\[
\delta(u) \geq \frac1{5(n+2)} \int_{\R^n} |\na u_o|^2 dx.
\]
This proves \eqref{eq:oddpart}. We have
\[
\int_{\R^n} |\na u_e|^2 dx = 1 - \int_{\R^n} |\na u_o|^2 dx \geq 1 - 5(n+2) \de(u) \geq \frac12
\]
here we use the assumption $\de(u) \leq \frac1{10(n+2)}$. Using again \eqref{eq:estodd}, we get
\[
\delta(u) \geq \frac12 \delta(u_e).
\]
This proves \eqref{eq:evenpart}.
\end{proof}

\subsection{Estimate for the even functions}
We start this subsection by establishing an identity that provides an alternative proof of the second order uncertainty principle of Cazacu, Flynn and Lam.

\begin{lemma}\label{identity}
For any function $u\in C_0^\infty(\R^n)$, it holds
\begin{equation}\label{eq:crucial}
\irn |\Delta u|^2 dx + \irn |x|^2 u^2 dx = (n+2) \irn |\na u|^2 dx + \irn \|\na^2 v - x \otimes \na v\|_{HS}^2 e^{-|x|^2} dx,
\end{equation}
where $u = v e^{-\frac{|x|^2}2}$.
\end{lemma}

%We will give another proof of the result of Cazacu, Flynn, and Lam which provides us away to establish a stability estimate for their result. For a function $u \in C^2(\R^n)$, we denote by $\na^2 u$ its hessian matrix, and for nonzero vectors $x, y \in \R^n$, we denote by $x\otimes y$ the matrix whose $(i,j)-$ entry is $x_i y_j$, in other word $x\otimes y: w \to (y \cdot w) x$ for any $w \in \R^n$. For a $n\times n$ matrix $A$, we use $\|A\|_{HS}$ to denote its Hilbert - Schmidt norm, i.e., $\|A\|_{HS} =\sqrt{\text{\rm Tr}(A^t A)}$.
\begin{proof}
Firstly, by using integration by parts we can easily check that
\begin{align}\label{eq:Hessian1}
\irn \|\na^2 u + \na u \otimes x\|_{HS}^2 dx & = \int_{\R^n} \|\na^2 u\|_{HS}^2 dx + 2 \int_{\R^n} \na^2 u (\na u) \cdot x dx + \irn |\na u|^2 |x|^2 dx\notag\\
&= \int_{\R^n} |\Delta u|^2 + \irn \na(|\na u|^2) \cdot x dx + \irn |\na u|^2 |x|^2 dx\notag\\
&= \int_{\R^n} |\Delta u|^2 - n \irn |\na u|^2 dx + \irn |\na u|^2 |x|^2 dx.
\end{align}
Next, we set $u = v e^{-|x|^2/2}$. Then, $v\in C_0^\infty(\R^n)$ and we have
$$\na^2 u = [\na^2 v -\na v \otimes x - x \otimes \na v - v(I_n -x \otimes x)] e^{-\frac{|x|^2}2}$$
and
$$\na u \otimes x = (\na v \otimes x - x \otimes x v) e^{-\frac{|x|^2}2}.$$
Consequently, it holds
$$\na^2 u + \na u \otimes x = [\na^2 v - x \otimes \na v - v I_n] e^{-\frac{|x|^2}2}$$
which yields
\begin{equation}\label{eq:newa}
\|\na^2 u + \na u \otimes x\|_{HS}^2 = \|\na^2 v - x \otimes \na v\|_{HS}^2 e^{-|x|^2} -2(\Delta v - \na v\cdot x)v e^{-|x|^2} + n v^2 e^{-|x|^2}.
\end{equation}
By using again integration by parts, we get
\begin{align*}
\int_{\R^n} (\Delta v - \na v \cdot x) v e^{-|x|^2} dx &= -\irn |\na v|^2 e^{-|x|^2} dx + \irn \na v \cdot x v e^{-|x|^2} dx \\
&= -\irn |\na v|^2 e^{-|x|^2} dx + \frac12 \irn \na v^2 \cdot x e^{-|x|^2} dx\\
&=-\irn |\na v|^2 e^{-|x|^2} dx -\frac n2 \irn v^2 e^{-|x|^2} dx + \irn v^2 |x|^2 e^{-|x|^2} dx.
\end{align*}
Integrating both sides of \eqref{eq:newa} in $\R^n$ and using the preceding equality, we arrive at
\begin{align}\label{eq:Hessian2}
\irn \|\na^2 u + \na u \otimes x\|_{HS}^2 dx &= \irn \|\na^2 v - x \otimes \na v\|_{HS}^2 e^{-|x|^2}dx + 2\irn |\na v|^2 e^{-|x|^2} dx \notag\\
&\qquad +2n \irn v^2 e^{-|x|^2} dx -2 \irn v^2 |x|^2 e^{-|x|^2} dx.
\end{align}
Finally, by using integration by parts again, we have
\begin{align}\label{eq:Hessian3}
\irn |\na u|^2 dx & =\irn |\na v - x v|^2 e^{-|x|^2}dx\notag\\
&= \irn |\na v|^2 e^{-|x|^2} dx - 2 \irn \na v \cdot x v e^{-|x|^2} dx + \irn |x|^2 v^2 e^{-|x|^2} dx \notag\\
&= \irn |\na v|^2 e^{-|x|^2} dx -\irn \na v^2 \cdot x e^{-|x|^2} dx + \irn |x|^2 v^2 e^{-|x|^2} dx\notag\\
&= \irn |\na v|^2 e^{-|x|^2} dx + n \irn v^2 e^{-|x|^2} dx - \irn |x|^2 v^2 e^{-|x|^2} dx.
\end{align}
The identity \eqref{eq:crucial} is now followed from \eqref{eq:Hessian1}, \eqref{eq:Hessian2} and \eqref{eq:Hessian3}.
\end{proof}

As a remark, we show that \eqref{eq:crucial} provides another proof of \eqref{eq:CFL}. Indeed, from \eqref{eq:crucial} we have
\begin{equation}\label{eq:crucial*}
\irn |\Delta u|^2 dx + \irn |x|^2 u^2 dx \geq (n+2) \irn |\na u|^2 dx,
\end{equation}
for any function $u\in C_0^\infty(\R^n)$. Applying \eqref{eq:crucial*} for function $u_\lam(x) = \lam^{n/2-1} u(\lam x)$ with $\lam >0$, we get
$$\lam^2 \irn |\Delta u|^2 dx + \lam^{-2} \irn |x|^2 u^2 dx \geq (n+2) \irn |\na u|^2 dx$$
for any $\lam >0$. Choosing $\lam = (\irn |\Delta u|^2 dx/ \irn |x|^2 u^2 dx)^{\frac14}$ implies \eqref{eq:CFL}.

Our next task is to use \eqref{eq:crucial} to prove a stability estimate for \eqref{eq:CFL} on even functions. We first prove the following result

\begin{theorem}\label{evenstab}
Let $v\in C_0^\infty(\R^n)$ be an even function. Then it holds
\begin{equation}\label{eq:remainder}
			\irn \|\na^2 v - x \otimes \na v\|_{HS}^2 e^{-|x|^2} dx\geq \frac{4n}{n+2}\left(\irn |\na ((v-c)e^{-\frac{|x|^2}{2}})|^2dx\right),
		\end{equation}
		where
		\[
		c= \frac{\int_{\R^n} v e^{-|x|^2} dx}{\int_{\R^n} e^{-|x|^2} dx}.
		\]
\end{theorem}
Before proving Theorem \ref{evenstab}, we provide some preparations.
Note that
\[
\na[(v-c) e^{-\frac{|x|^2}2}] = [\na(v- c) -x(v-c) ]e^{-\frac{|x|^2}2},
\]
then by using integration by parts, we get
\begin{multline*}
\irn |\na ((v-c)e^{-\frac{|x|^2}{2}})|^2dx = \irn |\na(v-c)|^2 e^{-|x|^2}dx + n\irn (v-c)^2 e^{-|x|^2}dx \\- \irn (v-c)^2 |x|^2 e^{-|x|^2}dx.
\end{multline*}
Notice that by making the change of variable $x = y/\sqrt 2$ and making the change of function $w(y) = v(y/\sqrt 2)$, then the inequality \eqref{eq:remainder} is equivalent to
\begin{align}\label{eq:remainder1}
			&\irn \|2\na^2 w - y \otimes \na w\|_{HS}^2 d\gamma_n\notag\\
&\quad \geq \frac{2n}{n+2}\left(4\irn |\na(w-c)|^2d\gamma_n+2n\irn (w-c)^2d\gamma_n-\irn |y|^2(w-c)^2d\gamma_n\right),
		\end{align}
where $\gamma_n$ denotes the standard Gaussian measure on $\R^n$, i.e, 
\[
d \gamma_n(y) = \frac{e^{-\frac{|y|^2}2}}{(2 \pi)^{\frac n2}} dy.
\]
The proof of \eqref{eq:remainder1} is based on the spectral analysis of the Ornstein-Uhlenbeck operator on the Gaussian space. 
Let $L_n$ denote the Ornstein-Uhlenbeck operator on the Gaussian space, i.e., $L_n w(y) = \Delta w(y) - \na w(y) \cdot y$. It holds that
\begin{equation}\label{eq:IBP}
\irn L_n w_1\cdot w_2 d\ga_n = -\irn \na w_1 \cdot \na w_2 d\ga_n.
\end{equation}
It is well-known that $L^2(\gamma_1)$ has an orthogonal basis consisting of the Hermite polynomials $H_0, H_1, H_2, \ldots$ with 
$$H_i(t) = (-1)^i e^{t^2/2} \frac{d^i}{dt^i} (e^{-t^2/2}),\qquad i=0,1,2,\ldots.$$
Note that $H_i$ is a polynomial of degree  $i$, $L_1 H_i = -i H_i$ and $\int_{\R} H_i^2 d\gamma_1 = i!$. For examples, $H_0\equiv 1, H_1(t) = t, H_2(t) = t^2 -1, H_3(t) = t^3 -3t,\ldots$. The Hermite polynomials have the following properties (see \cite[formulas $22.7.14$ and $22.8.8$]{Abra})
\begin{equation*}%\label{eq:Her1}
H_{i+1}(t) =t H_i(t) - iH_{i-1}(t),\quad { i\geq 1},
\end{equation*}
and
\begin{equation*}%\label{eq:Her2}
H_i'(t) = i H_{i-1}(t),\quad i\geq 1.
\end{equation*}
These properties imply the following results,
\begin{proposition}\label{nhanx2}
There holds
$$t^2 H_0(t) = H_2(t) + H_0(t),\, t^2 H_1(t) = H_3(t) + 3H_1(t)$$
and
$$ t^2 H_i(t) = H_{i+2}(t) + (2i +1) H_i(t) + i(i-1) H_{i-2}(t), \quad i\geq 2.$$
\end{proposition}
Using Proposition \ref{nhanx2}, we can easily to show for any $i\leq j$ that
\begin{equation}\label{eq:integralHix2}
\int_{\R} t^2 H_i(t) H_j(t) d\ga_1(t) =
\begin{cases}
0&\mbox{if $j =i+1$ or $j > i+2$}\\
(2i+1) i! &\mbox{if $i =j$}\\
(i+2)! &\mbox{if $j = i+2$.}
\end{cases}
\end{equation}
For each $I \in \Z_+^n$, $I = (i_1,\ldots,i_n), i_k \geq 0, k =1,\ldots,n$, we denote 
$$H_I(y) = H_{i_1}(y_1)\cdots H_{i_n}(y_n).$$
So,  $\{H_I\}_{I\in \Z_+^n}$ forms an orthogonal basis of $L^2(\gamma_n)$ and 
$$L_n H_I = -|I| H_I,\quad \text{\rm with} \quad |I| = i_1 + i_2 + \cdots + i_n.$$
Notice that $\int_{\R^n} H_I^2 d\gamma_n = i_1! i_2! \cdots i_n!=: I!$ for $I = (i_1,i_2,\ldots i_n) \in \Z_n^+$. For $I, J \in \Z^n_+$, we say that $I$ is a neighborhood of $J$ (denote by $I \sim J$) if and only if there exists uniquely a $1\leq k\leq n$ such that $|j_k -i_k| =2$ and $j_l = i_l$ for any $l\not =k$. We say that $I \leq J$ if $i_k \leq j_k$ for any $k=1,2,\ldots,n$. We shall use \eqref{eq:integralHix2} to prove the next lemma.%, it is easy to show for any $I, J \in \Z_n^+$ that
\begin{lemma}\label{HIJ2}
For any $I, J\in \Z_n^+$, it holds,
\begin{equation}\label{eq:HI2}
\int_{\R^n} H_I(y) H_J(y) |y|^2 d\gamma_n(y) = \begin{cases}
(2|I| + n) I ! &\mbox{if $I = J$}\\
J!&\mbox{if $I\leq J$ and $I\sim J$}\\
I! &\mbox{if $J \leq I$ and $J \sim I$}\\
0 &\mbox{otherwise.}
\end{cases}
\end{equation}
\end{lemma}
\begin{proof}
We first consider the case $I =J$. Using \eqref{eq:integralHix2}, we have
		\begin{align*}
			\irn H_I^2|y|^2d\ga_n&=\sum_{k=1}^n\irn H_I^2y_k^2d\ga_n	=\sum_{k=1}^n\left(\int_{\R} H_{i_k}^2y_k^2d\ga_1\right)\prod_{l\not=k}\int_{\R}H_{i_j}^2d\ga_1\\
			&	=\sum_{k=1}^n(2i_k+1)i_k!\prod_{l\not=k}i_l!=(2|I|+n)I!.
		\end{align*}
		
We next consider the case $I \leq J$ and $I \sim J$. So there is an index $l$ such that $j_l = i_l + 2$ and $j_k = i_k$ for any $k \not =l$. Using \eqref{eq:integralHix2} and $\int_{\R} H_{i_l} H_{i_l +2} d\ga_1 = 0$, we have
\begin{align*}
				\irn H_IH_J|y|^2d\ga_n&=\sum_{p=1}^n\irn H_{i_l}(y_l)H_{i_l+2}(y_l)\lt(\prod_{k\not=l}H_{i_k}^2\rt)y_p^2d\ga_n\\
			&=\int_{\R} H_{i_l}(y_l)H_{i_l+2}(y_l) y_l^2d\ga_1\prod_{k\not=l}\int_{\R} H_{i_k}^2d\ga_1\\
&\qquad +\sum_{p\not=l}\left(\int_{\R} H_{i_l}H_{i_l+2}d\ga_1\right)\left(\int_{\R^{n-1}} \left(\prod_{k\not=l}H_{i_k}^2\right)y_p^2d\ga_{n-1}\right)\\
			&=(i_l+2)!\prod_{k\not=l}i_k!=J!
\end{align*}

The case $J \leq I$ and $J\sim$ is treated similarly.

Finally, we consider the case $I \not = J$ and $I \not\sim J$. In this case, there are two subcases:
		
		$\bullet$	If there exist at least two indices $i_k\not=j_k$ and $i_l\not=j_l$ ($k\not=l$), then 
		\begin{align*}
			\irn H_IH_J|y|^2d\ga_n=\sum_{p=1}^n\irn H_IH_Jy_p^2d\ga_n=0
		\end{align*}
		since for each $p$, in the decomposition of $\irn H_IH_Jy_p^2d\ga_n$,  there is always at least one factor of the form $\int_{\R}H_{i_l}H_{j_l}d\ga_1=0$ or $\int_{\R}H_{i_k}H_{j_k}d\ga_1=0$.
		
		$\bullet$ If there is only one index $i_k\not=j_k$, then $|i_k-j_k|\not=2$ since $I\not\sim J$. Using \eqref{eq:integralHix2}, we have $\int_{\R} H_{i_k}H_{j_k}y_k^2d\ga_1=0$. Remark that $\int_{\R} H_{i_k}H_{j_k}d\ga_1=0$. From these facts, we get
		\begin{align*}
				\irn H_IH_J|y|^2d\ga_n&=\sum_{p=1}^n\irn H_IH_Jy_p^2d\ga_n\\
			&=\int_{\R} H_{i_k}H_{j_k}y_k^2d\ga_1\left(\prod_{l\not=k}\int_{\R} H_{i_l}^2d\ga_1\right)\\
&\qquad +\sum_{p\not=k}\left(\int_{\R} H_{i_k}H_{j_k}d\ga_1\right)\left(\int_{\R^{n-1}} \left(\prod_{l\not=k}H_{i_l}^2\right)y_p^2d\ga_{n-1}\right)\\
&=0.
		\end{align*}
		
		%\noindent{\bf Case 3: $I\sim J, j_l=i_l+2$ .} We have $\int_{\R} H_{i_k}H_{i_l+2}d\ga_1$ and then
		
\end{proof}

Let  $\{e_1,e_2,...,e_n\}$ be the canonical basis of $\R^n$. The next lemma is an application of Fubini theorem for sums.
\begin{lemma}\label{Fubini}
Let $\{c_{I}\}$ be a nonnegative sequence on $\Z_n^+$ and $f, g$ be two nonnegative functions on $\Z$. Then we have
\[
\sum_{|I|\geq 2} f(|I|) \sum_{l=1}^n c_{I+2e_l} g(i_l +2) = \sum_{|I|\geq 4} c_I f(|I|-2) \sum_{k: i_k\geq 2} g(i_k).
\]
\end{lemma}
\begin{proof}
Note that for each $J \geq I$ and $J\sim I$, there exists a unique $l \in \{1,2,\ldots, n\}$ such that $J = I + 2e_l$. Hence, it hols
\[
\sum_{l=1}^n c_{I+2e_l} g(i_l +2) = \sum_{J}\sum_{h=1}^n c_J \chi_{\{K: K\geq I, K\sim I\}}(J) \chi_{\{k: j_k -i_k =2\}}(h) g(j_h)
\]
here, $\chi_A$ denotes the characteristic function of the set $A$. Using Fubini theorem, we have
\begin{align*}
\sum_{|I|\geq 2} f(|I|) \sum_{l=1}^n c_{I+2e_l} g(i_l +2) &= \sum_{|I|\geq 2} \sum_{J}\sum_{h=1}^n f(|I|) c_J \chi_{\{K: K\geq I, K\sim I\}}(J) \chi_{\{k: j_k -i_k =2\}}(h) f(j_h)\\
&=\sum_{J} \sum_{|I|\geq 2} \sum_{h=1}^n f(|I|) c_J \chi_{\{K: K\geq I, K\sim I\}}(J) \chi_{\{k: j_k -i_k =2\}}(h) f(j_h)\\
&= \sum_{J} \sum_{|I|\geq 2, I \leq J, I\sim J} \sum_{h=1}^n f(|I|) c_J \chi_{\{k: j_k -i_k =2\}}(h) f(j_h)\\
&=\sum_{|J|\geq 4} f(|J|-2) c_J \sum_{|I|\geq 2, I \leq J, I\sim J}\sum_{h=1}^n \chi_{\{k: j_k -i_k =2\}}(h) f(j_h),
\end{align*}
here we use the fact that if $I\leq J$ and $I\sim J$ then $|I| = |J| -2$. Note that for a fixed $J$ with $|J| \geq 4$, if there is an index $l$ such that $j_l \geq 2$ then $I = (j_1, \ldots, j_l -2, \ldots, j_n)$ satisfies $|I| \geq 2, I \leq J$ and $I \sim J$. Thus, we have
\[
\sum_{|I|\geq 2, I \leq J, I\sim J}\sum_{h=1}^n \chi_{\{k: j_k -i_k =2\}}(h) f(j_h) = \sum_{h: j_h \geq 2} f(j_h).
\]
Inserting this equality in the preceding one, we complete the proof of this lemma.
\end{proof}
We are now ready to prove Theorem \ref{evenstab}
\begin{proof}[Proof of Theorem \ref{evenstab}]
As mentioned before, it is enough to prove \eqref{eq:remainder1}. We remark that
\[
c = \frac{\int_{\R^n} v e^{-|x|^2} dx}{\int_{\R^n} e^{-|x|^2} dx} = \int_{\R^n} w(y) d\gamma_n(y).
\]
We decompose $w$ in the basis $\{H_I\}_{I \in \Z_n^+}$ as 
$$w = \sum_{I\in \Z_+^n} a_I  H_I,$$
with 
\[
a_I = \frac1{I!} \int_{\R^n} w(y) H_I(y) d\ga_n(y).
\]
Notice that $c = a_{(0,\ldots,0)}$ and $a_I =0$ for any $I$ with $|I| =1$ since $w$ is even.

We first compute 
\[
4\irn |\na(w-c)|^2d\gamma_n+2n\irn (w-c)^2d\gamma_n-\irn |y|^2(w-c)^2d\gamma_n.
\]
Denote $\bar w = w -c$, we have
\[
\bar w = \sum_{I : |I|\geq 2} a_I H_I.
\]
This implies that 
		\begin{align}\label{e060120224}
			\begin{split}
				\irn |\na\bar w|^2d\ga_n&=\sum_{|I|\geq 2}|I| b_I^2,\\
				\irn \bar w ^2 d\ga_n &= \sum_{|I| \geq 2} b_I^2,\\
			\end{split}
		\end{align}
here we denote $b_I = a_I \sqrt{I!}$. Using Lemma \ref{HIJ2}, we have
\begin{align}\label{e060120227}
		\irn \bar w^2|y|^2d\ga_n & =\sum_{|I|\geq 2}a_I^2\irn H_I^2|y|^2d\ga_n+\sum_{I\not= J, |I|, |J|\geq 2}a_I a_J\irn H_IH_J|y|^2d\ga_n\notag\\
				&=\sum_{|I|\geq 2}(2|I|+n)b_I^2+2\sum_{|I|\geq 2}a_I\sum_{J \geq I, J\sim I}a_JJ!.
		\end{align}

From \eqref{e060120224} and \eqref{e060120227} and Lemma \ref{Fubini}, we obtain 
		\begin{align}\label{e070120221}
		&	4\irn |\na\bar w|^2d\gamma_n+2n\irn \bar w^2d\gamma_n-\irn |y|^2\bar w^2d\gamma_n\notag\\
&=\sum_{|I|\geq 2}(2|I|+n)b_I^2-2\sum_{|I|\geq 2}a_I\sum_{J\sim I, J\geq I}a_JJ!\notag\\
			&=\sum_{|I|\geq 2}(2|I|+n)b_I^2-2\sum_{|I|\geq 2}b_I\sum_{l=1}^nb_{I+2e_l}\sqrt{(i_l+1)(i_l+2)}\notag\\
			&=\sum_{|I|\geq 2}\sum_{l=1}^n\left((i_l+1)b_I^2-2b_I\sqrt{i_l+1}b_{I+2e_l}\sqrt{i_l+2}+b_{I+2e_l}^2(i_l+2)\right)\notag\\
			&+\sum_{|I|\geq 2}|I|b_I^2-\sum_{|I|\geq 2}\sum_{l=1}^nb_{I+2e_l}^2(i_l+2)\notag\\
			&=\sum_{|I|\geq 2}\sum_{l=1}^n\left(\sqrt{i_l+1}b_I -\sqrt{i_l+2}b_{I+2e_l}\rt)^2 + \sum_{|I|\geq 2}|I|b_I^2-\sum_{|I|\geq 4}b_I^2\sum_{k, i_k\geq 2}i_k\notag\\
			&=\sum_{|I|\geq 2}\sum_{l=1}^n\left(\sqrt{i_l+1}b_I -\sqrt{i_l+2}b_{I+2e_l}\rt)^2 + \sum_{|I|\geq 2}|I|b_I^2-\sum_{|I|\geq 4}b_I^2\lt(|I| -\sum_{k, i_k< 2}i_k\rt)\notag\\
			&= \sum_{|I|\geq 2}\sum_{l=1}^n\left(\sqrt{i_l+1}b_I -\sqrt{i_l+2}b_{I+2e_l}\rt)^2 + \sum_{2\leq |I|\leq 3}|I| b_I^2  +\sum_{|I|\geq 4}b_I^2\sum_{k, i_k< 2}i_k.
		\end{align}	

We next compute the left hand side of \eqref{eq:remainder1}. Using \eqref{eq:IBP} and the commutator relation $L_n\partial_j f = \partial_j L_n f +(\partial_j f)^2$ we have
		\begin{align*}
			\irn \|2\na^2 w &- y \otimes \na w\|_{HS}^2 d\gamma_n\\
&=\irn \|2\na^2 \bar w - y \otimes \na \bar w\|_{HS}^2 d\gamma_n\\
			&=4 \sum_{j=1}^n \irn |\na \partial_j \bar w|^2 d\gamma_n -4\irn \na^2 \bar w (\na \bar w) \cdot y d\gamma_n + \irn |y|^2 |\na \bar w|^2 d\ga_n\\
&=-4\sum_{j=1}^n\irn L(\partial_j \bar w)\partial_j\bar wd\ga_n -2\irn \na(|\na \bar w|^2) \cdot y d\ga_n +\irn |\na \bar w|^2|y|^2d\ga_n\\
			&=\sum_{j=1}^n\left(-4\irn \partial_j L_n\bar w\partial_j\bar wd\ga_n-4\irn (\partial_j\bar w)^2d\ga_n\right) +2n\irn |\na \bar w|^2d\ga_n\\
&\qquad\qquad\qquad -\irn |y|^2|\na \bar w|^2d\ga_n\\
			&=-4\irn \nabla L_n\bar w\cdot\nabla \bar wd\ga_n-4\irn |\na \bar w|^2d\ga_n +2n\irn |\na \bar w|^2d\ga_n-\irn |y|^2|\na \bar w|^2d\ga_n\\
			&= 4\irn (L_n \bar w)^2 d\ga_n + 2(n-2) \irn |\na \bar w|^2 d\ga_n - \irn |y|^2|\na w|^2d\ga_n.
		\end{align*}

		An elementary computation gives 
		\begin{equation*}
			\irn |y|^2|\na \bar w|^2d\ga_n=-\irn \bar wL_n \bar w|y|^2d\ga_n+n\irn \bar w^2d\ga_n-\irn |y|^2\bar w^2d\ga_n.
		\end{equation*}
		Substituting this into previous equality, we arrive at 
		\begin{align}\label{e0601202233}
			\irn \|2\na^2  w - y \otimes \na  w\|_{HS}^2 d\gamma_n&=4\irn (L_n\bar w)^2d\ga_n+2(n-2)\irn |\na \bar w|^2d\ga_n+\irn \bar wL_n \bar w|y|^2d\ga_n\notag\\
			&-n\irn \bar w^2d\ga_n+\irn |y|^2\bar w^2d\ga_n.
		\end{align}
It holds that
\[
L_n \bar w = \sum_{I: |I|\geq 2} -a_I |I| H_I.
\]
This implies
\begin{equation}\label{eq:Lnw2}
\irn (L_n \bar w)^2 d\ga_n = \sum_{|I|\geq 2} |I|^2 b_I^2.
\end{equation}
Using again Lemma \ref{HIJ2}, we have
\begin{align}\label{eq:Lnw}
\irn \bar w(-L_n \bar w)|y|^2d\ga_n&=\sum_{|I|,|J|\geq 2}a_I|J|a_J\irn H_IH_J|y|^2d\ga_n\notag\\
&=\sum_{|I|\geq 2} a_I^2 |I| \irn H_I^2 |y|^2 d\ga_n + \sum_{I\not=J} a_I a_J |J| \irn H_I H_J |y|^2 d\ga_n\notag\\
&=\sum_{|I|\geq 2} |I|(2|I| + n) b_I^2 + \sum_{|I|\geq2}a_I \sum_{J\geq I, J\sim I} a_J |J| J! + \sum_{|J|\geq 2} a_J |J|\sum_{I \geq J, I \sim J} a_I I!\notag\\
&= \sum_{|I|\geq 2} |I|(2|I| + n) b_I^2 + 2\sum_{|I|\geq2}a_I (|I| +1) \sum_{J\geq I, J\sim I} a_J J.,
\end{align}

Plugging \eqref{eq:Lnw2}, \eqref{eq:Lnw} and \eqref{e060120227} into \eqref{e0601202233}, we get
		\begin{align*}%\label{e060120228}
			&\irn \|2\na^2 w - y \otimes \na w\|_{HS}^2 d\gamma_n\\
&=4\sum_{|I|\geq 2}|I|^2 b_I^2+2(n-2) \sum_{|I|\geq 2}|I| b_I^2 -\sum_{|I|\geq 2}|I|(2|I|+n)b_I^2-2\sum_{|I|\geq 2}a_I(|I|+1)\sum_{J\sim I, J\geq I}a_JJ!\notag\\
			&-n\sum_{|I|\geq 2}b_I^2+\sum_{|I|\geq 2}(2|I|+n)b_I^2 +2\sum_{|I|\geq 2}a_I\sum_{J\sim I, J \geq I}a_JJ!\\
			&= \sum_{|I|\geq 2} |I| (2|I|+ n-2) b_I^2 - 2\sum_{|I|\geq 2}a_I|I|\sum_{J\geq I, J\sim I}a_JJ!.
		\end{align*}
Note that 
\[
\sum_{|I|\geq 2}a_I|I|\sum_{J\geq I, J\sim I}a_JJ! = \sum_{|I|\geq 2}b_I |I| \sum_{l=1}^n b_{I + 2e_l} \sqrt{(i_l +1)(i_l + 2)}.
\]
Using Lemma \ref{Fubini}, we obtain
\begin{align}\label{eq:compare1}
			&\irn \|2\na^2 w - y \otimes \na w\|_{HS}^2 d\gamma_n\notag\\
&=\sum_{|I|\geq 2}|I|(2|I|+n-2)b_I^2 -2\sum_{|I|\geq 2}b_I|I|\sum_{l=1}^nb_{I+2e_l}\sqrt{(i_l+1)(i_l+2)}\notag\\
			&=\sum_{|I|\geq 2}|I|\sum_{l=1}^n\left(\sqrt{i_l+1}b_I-\sqrt{i_l+2}b_{I+2e_l}\right)^2+\sum_{|I|\geq 2}|I|(|I|-2)b_I^2\notag\\
&\qquad\qquad-\sum_{|I|\geq 2}|I|\sum_{l=1}^nb_{I+2e_l}^2(i_l+2)\notag\\
			&=\sum_{|I|\geq 2}|I|\sum_{l=1}^n\left(\sqrt{i_l+1}b_I-\sqrt{i_l+2}b_{I+2e_l}\right)^2+\sum_{|I|\geq 2}|I|(|I|-2)b_I^2 \notag\\
&\qquad\qquad -\sum_{|I|\geq 4}b_I^2 (|I|-2) \sum_{k: i_k\geq 2} i_k\notag\\
&=\sum_{|I|\geq 2}|I|\sum_{l=1}^n\left(\sqrt{i_l+1}b_I-\sqrt{i_l+2}b_{I+2e_l}\right)^2+\sum_{|I|\geq 2}|I|(|I|-2)b_I^2 \notag\\
&\qquad\qquad -\sum_{|I|\geq 4}b_I^2 (|I|-2)(|I| -\sum_{k: i_k <2} i_k)\notag\\
&= \sum_{|I|\geq 2}|I|\sum_{l=1}^n\left(\sqrt{i_l+1}b_I-\sqrt{i_l+2}b_{I+2e_l}\right)^2+ 3\sum_{|I|= 2}b_I^2  + \sum_{|I|\geq 4}b_I^2 (|I|-2)\sum_{k: i_k <2} i_k\notag\\
&\geq 2\sum_{|I|\geq 2}\sum_{l=1}^n\left(\sqrt{i_l+1}b_I-\sqrt{i_l+2}b_{I+2e_l}\right)^2  + 2\sum_{|I|\geq 4}b_I^2\sum_{k: i_k <2} i_k
		\end{align}	
%The preceding inequality implies
%\begin{align}\label{eq:compare1}
%\irn \|2\na^2 w &- y \otimes \na w\|_{HS}^2 d\gamma_n\notag\\
%&\geq \sum_{|I|\geq 2}\sum_{l=1}^n\left(\sqrt{i_l+1}b_I -\sqrt{i_l+2}b_{I+2e_l}\rt)^2 + 3\sum_{|I|=3}b_I^2 +\sum_{|I|\geq 4}b_I^2\sum_{k, i_k< 2}i_k.
%\end{align}

On the other hand, we have
\begin{align}\label{e080120224}
			\irn \|2\na^2 w& - y \otimes \na w\|_{HS}^2 d\gamma_n\notag\\
			&=\sum_{|I|\geq 2}|I|(2|I|+n-2)b_I^2 -2\sum_{|I|\geq 2}b_I|I|\sum_{l=1}^nb_{I+2e_l}\sqrt{(i_l+1)(i_l+2)}\notag\\
			&=\sum_{|I|\geq 2}\left(|I|b_I-\sum_{l=1}^nb_{I+2e_l}\sqrt{(i_l+1)(i_l+2)}\right)^2+\sum_{|I|\geq 2}|I|(|I|+n-2)b_I^2\notag\\
			&-\sum_{|I|\geq 2}\left(\sum_{l=1}^nb_{I+2e_l}\sqrt{(i_l+1)(i_l+2)}\right)^2.
		\end{align}
Using the Cauchy-Schwarz inequality, we get
\begin{align*}%\label{e080120225}
			\left(\sum_{l=1}^nb_{I+2e_l}\sqrt{(i_l+1)(i_l+2)}\right)^2&\leq \left(\sum_{l=1}^n(i_l+1)\right)\left(\sum_{l=1}^nb_{I+2e_l}^2(i_l+2)\right)\\
			&=(|I|+n)\sum_{l=1}^nb_{I+2e_l}^2(i_l+2).
		\end{align*}
		This and Lemma \ref{Fubini} imply that 
		\begin{align}\label{e080120225}
			\sum_{|I|\geq 2}	\left(\sum_{l=1}^nb_{I+2e_l}\sqrt{(i_l+1)(i_l+2)}\right)^2&\leq \sum_{|I|\geq 2}(|I|+n)\sum_{l=1}^nb_{I+2e_l}^2(i_l+2)\notag\\
			&= \sum_{|J|\geq 4} b_J^2 (|J| + n -2) \sum_{j_k\geq 2} j_k\notag\\
			&=\sum_{|J|\geq 4}	b_J^2(|J|+n-2)\left(|J|-\sum_{j_k<2}j_k\right).
			%&=\sum_{|J|\geq 4}	b_J^2|J|(|J|+n-2)-\sum_{|J|\geq 4}b_J^2\left(\sum_{j_m<2}j_m\right).
		\end{align}
		Substituting \eqref{e080120225} into \eqref{e080120224}, we obtain
		\begin{align}\label{eq:compare2}
			\irn \|2\na^2 w - y \otimes \na w\|_{HS}^2 d\gamma_n&\geq \sum_{|I|\geq 2}\left(|I|b_I-\sum_{l=1}^nb_{I+2e_l}\sqrt{(i_l+1)(i_l+2)}\right)^2\notag\\
			&\qquad +\sum_{2\leq |I|\leq 3}|I|(|I|+n-2)b_I^2+\sum_{|J|\geq 4}b_J^2\left(\sum_{j_m<2}j_m\right)\notag\\
&\geq n\sum_{ 2\leq |I|\leq 3}|I|b_I^2.
		\end{align}		
The desired estimate \eqref{eq:remainder1} follows from \eqref{eq:compare1}, \eqref{eq:compare2} and \eqref{e070120221}.
\end{proof}

The next lemma given a stability estimate of \eqref{eq:CFL} for even functions.
\begin{lemma}\label{leven}
		Let $u\in C_0^\infty(\R^n)$ be a non-zero even function, then 
		\begin{equation}\label{e080120229}
			\delta(u)\geq \frac{4n}{(n+2)^2}\inf_{\varphi\in E}\left\{\frac{\irn |\nabla u-\nabla\varphi|^2dx}{\irn |\nabla u|^2dx}\right\}.
		\end{equation}
	\end{lemma}
	\begin{proof}[Proof of Lemma \ref{leven}]
	Since both sides of the inequality \eqref{e080120229} are invariant under the dilation and multiplication by a non-zero constant, then without loss of generality we can assume that  $\irn |\na u|^2 dx = 1$ and 
		$$\irn |\Delta u|^2 dx = \irn |x|^2 |\na u|^2 dx.$$
		Using the identity \eqref{eq:crucial}, we have 
		\begin{align*}
			\delta(u)& = \frac{\irn |\De u|^2 dx + \irn |x|^2 |\na u|^2 dx-(n+2)\irn |\na u|^2 dx}{n+2} \\
			& =\frac{1}{n+2}\irn \|\na^2 v - x \otimes \na v\|_{HS}^2 e^{-|x|^2}dx
		\end{align*}
		with $u=ve^{-\frac{|x|^2}{2}}$. Note that $v$ is also even. Let $c=\irn vd\mu_n$. Applying Theorem \ref{evenstab}, we obtain
		\begin{align*}
		\delta(u) \geq \frac{4 n}{(n+2)^2} \irn |\nabla [(v-c)e^{-\frac{|x|^2}2}]|^2 dx& = \frac{4 n}{(n+2)^2} \irn|\na u - \na (ce^{-\frac{|x|^2}2})|^2 dx\\
		&\geq  \frac{4 n}{(n+2)^2} \inf_{\varphi \in E} \irn |\na u -\na \varphi|^2 dx.
		\end{align*}

	\end{proof}

%Let $H^1(\mu_n)$ be the Sobolev space on $\R^n$ with respect to measure $\mu_n$.

\subsection{Proof of Theorem \ref{Stability}}
In this subsection, we given the proof of Theorem \ref{Stability} by using the results from two preceding subsections.
\begin{proof}[Proof of Theorem \ref{Stability}]
By the density argument, it is enough to prove Theorem \ref{Stability} for function $u\in C_0^\infty(\R^n)$. Without loss of generality, we can assume that $\irn |\na u|^2 dx =1$. Let 
\[
u_o(x) = \frac{u(x) -u(-x)}2,\quad u_e(x) = \frac{u(x) + u(-x)}2
\]
be the odd part and even part of $u$ respectively.

We first consider the case $\delta(u) \leq \frac1{96(n+2)}$. It follows from Theorem \ref{oddfunct} that
\[
\irn |\na u_o|^2 dx \leq 5(n+2) \delta(u)
\]
and $\delta(u_e) \leq 2 \delta(u)$. Applying Lemma \ref{leven} for $u_e$, we have
\[
\inf_{\varphi \in E} \irn |\na u_e - \na \varphi|^2 dx \leq \frac{(n+2)^2}{4n} \de(u_e) \irn |\na u_e|^2 dx \leq \frac{(n+2)^2}{2n} \de(u) \leq (n+2) \de(u),
\] 
since $n\geq 2$. Using the inequality
\[
\irn |\na u - \na \varphi|^2 dx \leq 2\lt(\irn |\na u_o|^2 dx + \irn |\na u_e -\na \varphi|^2 dx\rt),
\]
we obtain
\[
\inf_{\varphi \in E} \irn |\na u - \na \varphi|^2 dx \leq 12(n+2) \de(u).
\]
From the previous estimate, we get
\[
\inf_{v \in E} \irn |\na u -\na \varphi|^2 dx \leq \frac18.
\]
We now show that the infimum above is attained by a non-zero function $\varphi\in E$. Indeed, let $\varphi_i = c_i e^{-\lambda_i |x|^2/2}$ be a sequence in $E$ such that
$$\lim_{i\to \infty} \irn |\na u -\na \varphi_i|^2 dx = \inf_{\varphi \in E} \irn |\na u -\na \varphi|^2 dx \leq \frac18.$$
Using triangle inequality, we have
$$\Big(\irn |\na u -\na \varphi_i|^2 dx\Big)^{\frac12} \geq \Big| \Big(\irn |\na \varphi_i|^2 dx\Big)^{\frac12} -1\Big|.$$
Thus for $i$ large enough, we get
$$\frac14 \leq \irn |\na \varphi_i|^2 dx \leq \frac 94.$$
A simple computation gives
$$\irn |\na \varphi_i|^2 dx = c_i^2 \lam_i^{-\frac n2 + 1} \irn |x|^2 e^{-|x|^2} dx.$$
So there are $a, A >0$ such that $a^2 \leq c_i^2 \lam_i^{-\frac n2 + 1}\leq A^2$ for $i$ large enough. We also note that
\begin{equation}\label{eq:expan}
\irn |\na u -\na \varphi_i|^2 dx = 1 + c_i^2 \lam_i^{-\frac n2 + 1}\irn |x|^2 e^{-|x|^2} dx + 2c_i \lam_i^{\frac12} \irn \na u \cdot (\lam_i^{\frac12} x) e^{-|\lam_i^{\frac12} x|^2/2} dx.
\end{equation}
We next claim that $\lam_i$ is bounded from above and below by positive constants. Indeed, suppose, up to extract a subsequence, that
$$\lim_{i\to \infty} \lam_i = \infty.$$
For any $\epsilon >0$, there exists $R>0$ small enough such that $\int_{B_R} |\na u|^2 dx < \ep^2$. Hence, it holds
\begin{align*}
\Big|\irn \na u \cdot (\lam_i^{\frac12} x) e^{-|\lam_i^{\frac12} x|^2/2} dx\Big|&\leq \int_{B_R} |\na u| |\lam_i^{\frac12} x| e^{-|\lam_i^{\frac12} x|^2/2} dx + \int_{B_R^c} |\na u|^2 |\lam_i^{\frac12} x| e^{-|\lam_i^{\frac12} x|^2/2} dx\\
&\leq \Big(\int_{B_R^c} |\na u|^2 dx\Big)^{\frac12} \Big(\int_{B_R^c}|\lam_i^{\frac12} x|^2 e^{-|\lam_i^{\frac12} x|^2} dx\Big)^{\frac12}\\
&\qquad + \ep \Big(\int_{B_R}|\lam_i^{\frac12} x|^2 e^{-|\lam_i^{\frac12} x|^2} dx\Big)^{\frac12}\\
&\leq \lam_i^{-\frac n4}\Big(\int_{B_{R\sqrt{\lam_i}}^c}|x|^2 e^{-|x|^2} dx\Big)^{\frac12} + \lam_i^{-\frac n4} \ep  \Big(\int_{B_{R\sqrt{\lam_i}}}|x|^2 e^{-|x|^2} dx\Big)^{\frac12}
\end{align*}
which implies
$$
\Big|c_i \lam_i^{\frac 12}\irn \na u \cdot (\lam_i^{\frac12} x) e^{-|\lam_i^{\frac12} x|^2/2} dx\Big| \leq A \Big(\int_{B_{R\sqrt{\lam_i}}^c}|x|^2 e^{-|x|^2} dx\Big)^{\frac12} + A \ep  \Big(\int_{B_{R\sqrt{\lam_i}}}|x|^2 e^{-|x|^2} dx\Big)^{\frac12}.
$$
Let $i\to \infty$, we obtain
$$\limsup_{i\to \infty}\Big|c_i \lam_i^{\frac 12}\irn \na u \cdot (\lam_i^{\frac12} x) e^{-|\lam_i^{\frac12} x|^2/2} dx\Big| \leq A \Big(\irn |x|^2 e^{-|x|^2} dx\Big)^{\frac12} \ep.$$
Since $\ep>0$ is arbitrary, then 
$$\limsup_{i\to \infty}\Big|c_i \lam_i^{\frac 12}\irn \na u \cdot (\lam_i^{\frac12} x) e^{-|\lam_i^{\frac12} x|^2/2} dx \Big|=0.$$
This together with \eqref{eq:expan} implies
$$\frac18 \geq \lim_{i\to \infty} \irn |\na u -\na \varphi_i|^2 dx > 1$$
which is impossible.{ Concerning the lower bound of $\lambda_i$,}  suppose, up to extract a subsequence, that
$$\lim_{i\to \infty} \lam_i = 0.$$
For any $\epsilon >0$, there exists $R>0$ large enough such that $\int_{B_R^c} |\na u|^2 dx < \ep^2$. Hence, it holds
\begin{align*}
\Big|\irn \na u \cdot (\lam_i^{\frac12} x) e^{-|\lam_i^{\frac12} x|^2/2} dx\Big|&\leq \int_{B_R} |\na u| |\lam_i^{\frac12} x| e^{-|\lam_i^{\frac12} x|^2/2} dx + \int_{B_R^c} |\na u|^2 |\lam_i^{\frac12} x| e^{-|\lam_i^{\frac12} x|^2/2} dx\\
&\leq \Big(\int_{B_R} |\na u|^2 dx\Big)^{\frac12} \Big(\int_{B_R}|\lam_i^{\frac12} x|^2 e^{-|\lam_i^{\frac12} x|^2} dx\Big)^{\frac12}\\
&\qquad + \ep \Big(\int_{B_R^c}|\lam_i^{\frac12} x|^2 e^{-|\lam_i^{\frac12} x|^2} dx\Big)^{\frac12}\\
&\leq \lam_i^{-\frac n4}\Big(\int_{B_{R\sqrt{\lam_i}}}|x|^2 e^{-|x|^2} dx\Big)^{\frac12} + \lam_i^{-\frac n4} \ep  \Big(\int_{B_{R\sqrt{\lam_i}}^c}|x|^2 e^{-|x|^2} dx\Big)^{\frac12}
\end{align*}
which implies
$$
\Big|c_i \lam_i^{\frac 12}\irn \na u \cdot (\lam_i^{\frac12} x) e^{-|\lam_i^{\frac12} x|^2/2} dx\Big| \leq A \Big(\int_{B_{R\sqrt{\lam_i}}}|x|^2 e^{-|x|^2} dx\Big)^{\frac12} + A \ep  \Big(\int_{B_{R\sqrt{\lam_i}}^c}|x|^2 e^{-|x|^2} dx\Big)^{\frac12}.
$$
Let $i\to \infty$, we obtain
$$\limsup_{i\to \infty}\Big|c_i \lam_i^{\frac 12}\irn \na u \cdot (\lam_i^{\frac12} x) e^{-|\lam_i^{\frac12} x|^2/2} dx\Big| \leq A \Big(\irn |x|^2 e^{-|x|^2} dx\Big)^{\frac12} \ep.$$
Since $\ep>0$ is arbitrary, then 
$$\limsup_{i\to \infty}\Big|c_i \lam_i^{\frac 12}\irn \na u \cdot (\lam_i^{\frac12} x) e^{-|\lam_i^{\frac12} x|^2/2} dx \Big|=0.$$
This together with \eqref{eq:expan} implies
$$\frac18 \geq \lim_{i\to \infty} \irn |\na u -\na \varphi_i|^2 dx > 1$$
which is again impossible. So the claim is proved. As a consequence, there are $a_1, a_2, a_3 >0$ such that 
$$|c_i|\leq a_1,\qquad 0< a_2 \leq \lam_i \leq a_3$$
for any $i$. Extracting a subsequence, we can assume that $c_i \to c$ and $\lam_i \to \lam \in [a_2,a_3]$. Denote $\varphi = c e^{-\lam |x|^2/2} \in E $. We then have $\na \varphi_i \to \na \varphi$ in $L^2$ and hence
$$ \irn |\na u- \na \varphi|^2 dx = \lim_{i\to \infty} \irn |\na u -\na \varphi_i|^2 dx = \inf_{\varphi \in E} \irn |\na u -\na \varphi|^2 dx.$$
Remark that 
$$\irn |\na u|^2 dx =1,\quad \irn |\na u- \na \varphi|^2 dx = \inf_{\varphi \in E} \irn |\na u -\na \varphi|^2 dx \leq \frac 18$$ 
then $\varphi\not\equiv 0$. So we can choose a positive constant $a$ such that 
$$\irn |\na (a \varphi)|^2 dx = \irn |\na u|^2 dx = 1$$
or, equivalently $a = (\irn |\na \varphi|^2 dx)^{-\frac12}$. By the triangle inequality, we have
\begin{align*}
\Big|\Big(\irn |\na \varphi|^2 dx\Big)^{\frac12} -1\Big| &= \Big|\Big(\irn |\na \varphi|^2 dx\Big)^{\frac12} -\Big(\irn |\na u|^2 dx\Big)^{\frac12}\Big|\\
&\leq \Big(\irn |\na u-\na \varphi|^2 dx\Big)^{\frac12} \\
&\leq \sqrt{12(n+2) \delta(u)}. %\quad \Big(\leq \frac{1}{2\sqrt{2}}\Big).
\end{align*}
Using this estimate and the Cauchy-Schwartz inequality, we obtain
\begin{align*}
\irn |\na u - \na(a\varphi)|^2 dx &= \irn |\na u -\na \varphi + (1-a)\na \varphi|^2 dx\notag\\
&\leq 2 \irn |\na u -\na \varphi|^2 dx + 2 (a-1)^2 \irn |\na \varphi|^2 dx\notag\\
&\leq 24(n+2) \delta(u) + 2 \Big|\Big(\irn |\na \varphi|^2 dx\Big)^{\frac12} -1\Big|^2\notag\\
&\leq 48(n+2) \delta(u).
\end{align*}
This gives
\begin{equation}\label{eq:smalldeficit}
\delta(u) \geq \frac{1}{48(n+2)} \inf\Big\{\frac{\irn |\na u -\na \varphi|^2 dx}{\irn |\na u|^2 dx} \, :\, \varphi\in E, \irn |\na u|^2 dx = \irn |\na \varphi|^2 dx \Big\},
\end{equation}
when $\delta (u) \leq \frac{1}{96(n+2)}$.

We next prove the theorem for $\delta (u) > \frac{1}{96(n+2)}$. Since we always have
$$\inf\Big\{\frac{\irn |\na u -\na \varphi|^2 dx}{\irn |\na u|^2 dx} \, :\, \varphi \in E, \irn |\na u|^2 dx = \irn |\na \varphi|^2 dx \Big\} \leq 4$$
then
\begin{equation}\label{eq:largedeficit}
\delta(u) \geq \frac{1}{384(n+2)} \inf\Big\{\frac{\irn |\na u -\na \varphi|^2 dx}{\irn |\na u|^2 dx} \, :\, \varphi \in E, \irn |\na u|^2 dx = \irn |\na \varphi|^2 dx \Big\}.
\end{equation}
The inequality \eqref{eq:Stabestimate} follows immediately from \eqref{eq:smalldeficit} and \eqref{eq:largedeficit}.
\end{proof}

\section*{Acknowledgments}
This work was initiated and done when the second author visit Vietnam Institute for Advanced Study in Mathematics (VIASM) in 2020. He would like to thank the institute for hospitality and support during the visit.
%\bibliographystyle{abbrv}
%\bibliography{References__UP}
%\begin{thebibliography}{9999}
%\end{thebibliography}

\end{document}